\documentclass[12pt]{amsart}

\usepackage{amsmath, amscd, graphicx, latexsym, hyperref, rlepsf, times, gensymb}

\oddsidemargin .25in \theoremstyle{plain}

\newtheorem{Thm}{Theorem}

\newtheorem{Lem}[Thm]{Lemma}

\newtheorem{Con}[Thm]{Conjecture}

\newtheorem{Cor}[Thm]{Corollary}

\newtheorem{Prop}[Thm]{Proposition}

\newtheorem{Def}[Thm]{Definition}

\newtheorem{Rem}[Thm]{Remark}

\newtheorem{Ex}[Thm]{Example}

\usepackage{mathtools}
\usepackage{xcolor,enumerate}
\usepackage{blkarray}
\usepackage{color}
\usepackage{tikz}
\newcommand*\circled[1]{\tikz[baseline=(char.base)]{
            \node[shape=circle,draw,inner sep=2pt] (char) {#1};}}
\usepackage{enumitem}
\usepackage{graphicx}
\newcommand\sbullet[1][.5]{\mathbin{\vcenter{\hbox{\scalebox{#1}{$\bullet$}}}}}
\newcommand{\red}[1]{\textcolor{red}{\bf #1}}
\newcommand{\blue}[1]{\textcolor{blue}{\bf #1}}

\usepackage{palatino}
\usepackage[hypcap=false]{caption}
\usepackage{caption, subcaption, pinlabel}
\usepackage{color}
\usepackage{soul}
\usepackage{xcolor}

\newcommand{\hgt}{\operatorname{ht}}
\newcommand{\dpt}{\operatorname{dpt}}
\newcommand{\rbd}{\operatorname{RBD}}

\newcommand{\hj}{\operatorname{HJ}}

\newcommand{\OB}{\mathcal{OB}}

\def\a{\alpha}
\def\b{\beta}
\def\d{\delta}

\def\D{\Delta}
\def\g{\gamma}

\begin{document}

\title{Rational blowdown graphs for  \linebreak symplectic   fillings of lens spaces}

\author{Mohan Bhupal and Burak Ozbagci}


\address{Department of Mathematics,  METU, Ankara, Turkey, \newline bhupal@metu.edu.tr}
\address{Department of Mathematics, Ko\c{c} University, Istanbul, Turkey
\newline bozbagci@ku.edu.tr}


\begin{abstract}

In a previous work, we proved that each minimal symplectic  filling of any oriented lens space, viewed as the singularity link of some cyclic quotient singularity and equipped with its canonical contact structure,  can be obtained from the minimal resolution of the singularity by a sequence of  symplectic rational blowdowns along linear plumbing graphs.
Here we give a dramatically simpler visual presentation of our rational blowdown algorithm in terms of the triangulations of a convex polygon. As a consequence, we are able to organize the symplectic deformation equivalence classes of all minimal symplectic fillings of any given lens space equipped with its canonical contact structure, as a graded, directed, rooted, and connected graph, where the root is the minimal resolution of the corresponding cyclic quotient singularity and  each directed edge is a symplectic rational blowdown along an explicit linear plumbing graph. Moreover, we provide an upper bound for the {\em rational blowdown depth} of each minimal symplectic filling.
\end{abstract}

\maketitle

\section{Introduction}

For each  pair of coprime integers $(p,q)$ with $p > q \geq 1$,  the lens space $L(p,q)$ is
orientation preserving diffeomorphic to the link of some cyclic quotient singularity.   Let $\xi_{can}$  denote the canonical contact structure on $L(p,q)$, viewed as the singularity link. Lisca  \cite{l} classified the minimal symplectic fillings  of $(L(p,q), \xi_{can})$, up to  diffeomorphism.  These diffeomorphism classes are parametrized by a set $\mathcal{Z}_{k}(\textstyle{\frac{p}{p-q}})$  (see Section~\ref{sec: basic} for its definition)
of certain $k$-tuples of nonnegative integers, where $k$ is the length of the Hirzebruch-Jung continued fraction expansion of $\frac{p}{p-q}$. Moreover, each diffeomorphism class admits a unique symplectic structure, up to symplectic deformation equivalence. \cite{bono}.

Let $\mathcal{Z}_{k}$ denote the set of {\em admissible} $k$-tuples of  nonnegative integers which {\em represent zero}   (see Section~\ref{sec: basic} for its definition). As observed by Stevens  \cite{s}, the set $\mathcal{Z}_{k}$ can be identified with the set $\mathcal{T} (\mathcal{P}_{k+1})$ of all triangulations of a convex polygon $\mathcal{P}_{k+1}$ with $k+1$ vertices. By definition, $\mathcal{Z}_{k}(\textstyle{\frac{p}{p-q}})$  is a certain subset of $\mathcal{Z}_{k}$. It follows that the set of symplectic deformation classes of minimal symplectic fillings of $(L(p,q), \xi_{can})$  can be bijectively identified with a certain subset of $\mathcal{T}(\mathcal{P}_{k+1})$, which we denote by $\mathcal{T}^{p,q}(\mathcal{P}_{k+1})$ in this paper.

In \cite{bo}, we proved that, up to symplectic deformation equivalence, each minimal symplectic  filling of $(L(p,q), \xi_{can})$ can be obtained from the canonical symplectic filling, which is the minimal resolution of the corresponding cyclic quotient singularity, by a sequence of {\em symplectic}  rational blowdowns along {\em linear} plumbing graphs. To prove this result, we first constructed an explicit planar Lefschetz fibration on each minimal symplectic filling. Then we provided an algorithm so that, for each minimal symplectic filling, one can start with the monodromy factorization for the Lefschetz fibration on the minimal resolution and, by applying a sequence of lantern substitutions, obtain the monodromy factorization for the Lefschetz fibration on the minimal symplectic filling at hand. According to our algorithm, one has to allow achiral Lefschetz fibrations in the mid-sequence but the end of the sequence is always a (positive) Lefschetz fibration. Finally, we showed that for each minimal symplectic filling, the concatenation of these lantern substitutions is a sequence of symplectic  rational blowdowns along linear  plumbing graphs.

In order to show that our rational blowdowns are in fact symplectic (not just smooth) surgeries, we relied on the fact that each such monodromy substitution in the monodromy factorization of a Lefschetz fibration is a symplectic surgery, due to the work of Gay and Mark \cite{gm}.

Here we present our algorithm in terms of the triangulations of a convex polygon. The crucial observation is that each lantern substitution in the monodromy factorization of the corresponding planar Lefschetz fibration is realized by a diagonal flip move in the triangulations (see Section~\ref{sec: flip} for its definition), and therefore each rational blowdown which is obtained by a concatenation of lantern substitutions is realized by a sequence of diagonal flip moves. Moreover, with this new point of view, we are able to organize the symplectic deformation  equivalence classes of all minimal symplectic  fillings of $(L(p,q), \xi_{can})$ as a graded, directed, rooted, connected graph $\mathcal{G}^{p,q}_k$, where the root (meaning, the only vertex with no incoming edges) is the minimal resolution of the corresponding cyclic quotient singularity and  each directed edge is a symplectic rational blowdown along an explicit  linear plumbing graph. The grading is provided by the second Betti number of the minimal symplectic filling, where the minimal resolution has the highest grading.

\begin{Thm}\label{thm: main}  Let $(p,q)$ be a pair of coprime integers $(p,q)$ with $p > q \geq 1$, and let $k$ be the length of the Hirzebruch-Jung continued fraction expansion of $\frac{p}{p-q}$.  Then there is a graded, directed, rooted, connected graph $\mathcal{G}^{p,q}_k$, which we call the rational blowdown graph,  such that
\begin{enumerate}
\item{there is a bijection between the set of vertices of $\mathcal{G}^{p,q}_k$ and the set $\mathcal{T}^{p,q}(\mathcal{P}_{k+1})$ of certain triangulations of the convex polygon $\mathcal{P}_{k+1}$ with $k+1$ vertices, which parameterizes the symplectic deformation equivalence classes of minimal symplectic  fillings of $(L(p,q), \xi_{can})$, }
\item{the root vertex of $\mathcal{G}^{p,q}_k$ corresponds to  the initial triangulation, representing the minimal resolution,}
\item{each directed edge in $\mathcal{G}^{p,q}_k$ corresponds to a sequence of diagonal flips in the triangulations, which represents a symplectic rational blowdown along a linear plumbing graph, and}
\item{each vertex of $\mathcal{G}^{p,q}_k$ is graded by the second Betti number of the minimal symplectic filling it represents and each directed edge drops the grading by the number of diagonal flips used to construct that edge in item (3).}
\end{enumerate}
  \end{Thm}

\begin{Rem}{\em There is an elementary algorithm to obtain the linear plumbing graph for each directed edge of $\mathcal{G}^{p,q}_k$ described in item (3) of Theorem~\ref{thm: main}, based on the sequence of diagonal flips used to construct that edge. We formulated this algorithm as Proposition~\ref{prop: lpg}. } \end{Rem}

The graph $\mathcal{G}^{p,q}_k$ of Theorem~\ref{thm: main} is obtained from another graded, directed, rooted, connected graph, which we denote by $\mathcal{G}_k$. The set of vertices of $\mathcal{G}_k$ corresponds bijectively  to the set $\mathcal{T}(\mathcal{P}_{k+1})$ of all the triangulations of the convex polygon $\mathcal{P}_{k+1}$, and each {\em directed} edge connects two triangulations which differ only by a single diagonal flip along a distinguished diagonal.  We think of the vertices of $\mathcal{G}_k$ as light bulbs and the edges as the wires connecting the light bulbs. Once a pair $(p,q)$ is fixed as in Theorem~\ref{thm: main}, the graph $\mathcal{G}^{p,q}_k$  is essentially obtained from the graph $\mathcal{G}_k$ by turning on some of the light bulbs in $\mathcal{G}_k$ determined by $(p,q)$, and inserting new wires, if necessary,  to bypass the light-bulbs which are not turned on.

It is possible that the symplectic deformation type  of some minimal symplectic filling of $(L(p,q), \xi_{can})$ can be obtained from the minimal resolution by applying {\em distinct} sequences of symplectic rational blowdowns. This phenomenon is certainly reflected in our graph $\mathcal{G}^{p,q}_k$, as different possible paths (i.e., concatenations of the directed edges) from a vertex (in particular the root vertex) to any other are clearly visible in   $\mathcal{G}^{p,q}_k$.

 \begin{Def} A minimal symplectic filling of $(L(p,q), \xi_{can})$ is said to have  rational blowdown depth  $r$ if the minimal number of successive symplectic rational blowdowns along linear plumbing graphs needed to obtain the filling from the minimal resolution is equal to $r$, where the depth of the minimal resolution is set to be zero.
  \end{Def}

 \begin{Def} For $k \geq 3$, the depth of  $\textbf{n} = (n_1, \ldots, n_k) \in \mathcal{Z}_{k}$, denoted by $\dpt (\textbf{n})$, is the number of $1$'s in the interior of $\textbf{n}$, i.e., $\dpt (\textbf{n})$ is the cardinality of the set $\{ i \; |\;  1 < i < k \; \mbox{and} \; n_i=1\}$.
  \end{Def}

\begin{Prop} \label{prop: depth} Let $W_{p,q}(\textbf{n})$ denote the minimal symplectic filling of $(L(p,q), \xi_{can})$ that corresponds to $\textbf{n}= (n_1, \ldots, n_k)$ in the parameterizing set $\mathcal{Z}_{k}(\textstyle{\frac{p}{p-q}})$. Then  the rational blowdown depth  of  $W_{p,q}(\textbf{n})$  is bounded above by  $\dpt (\textbf{n})$. In particular, if $\dpt (\textbf{n})=1$, then $W_{p,q}(\textbf{n})$  is obtained from the minimal resolution by a single symplectic rational blowdown.
\end{Prop}

\begin{Con} \label{con: depth} The rational blowdown depth  of  $W_{p,q}(\textbf{n})$  is in fact equal to  $\dpt (\textbf{n})$. \end{Con}

In Propositions~\ref{prop: depth2} and ~\ref{prop: depth2a}, we give examples of minimal symplectic fillings  of rational blowdown depth $2$, and in Proposition~\ref{prop: depth3}, we give an example of  a  minimal symplectic filling  of rational blowdown depth $3$, all satisfying Conjecture~\ref{con: depth}. 

Notice that each Milnor fibre of any given cyclic quotient singularity is a Stein (and hence minimal symplectic)  filling of $(L(p,q), \xi_{can})$. By the work of Christophersen \cite{c} and Stevens \cite{s}, the set $\mathcal{Z}_{k}(\textstyle{\frac{p}{p-q}})$ also parameterizes these  Milnor fibres, up to diffeomorphism. As a matter of  fact,  Lisca \cite{l} proved that each diffeomorphism class of minimal symplectic fillings of $(L(p,q), \xi_{can})$ contains a  Stein representative and  proposed an {\em explicit}  one-to-one correspondence between the set of such Stein representatives and the set of Milnor fibres of the corresponding cyclic quotient singularity,   which was subsequently verified  by  N\'{e}methi and Popescu-Pampu \cite{npp}.

On the other hand, it is well-known that, for any  $(L(p,q), \xi_{can})$, the Milnor fibre of the Artin smoothing component of the corresponding cyclic quotient singularity gives a minimal symplectic filling which is symplectic deformation equivalent to the one obtained by deforming the symplectic structure on the minimal resolution (see \cite{bd}) of the singularity. The result below is an immediate consequence of the aforementioned one-to-one correspondence  of N\'{e}methi and Popescu-Pampu \cite{npp}.

\begin{Cor} Analogues of  Theorem~\ref{thm: main} and Proposition~\ref{prop: depth} hold when minimal symplectic fillings are replaced by Milnor fibres  of the corresponding cyclic quotient singularity and the minimal resolution is replaced with  the Milnor fibre of the Artin smoothing component.  \end{Cor}

\section{Continued fractions and  triangulations of a convex polygon} \label{sec: basic}

Suppose that $p > q \geq 1$ are coprime integers and let
$$\frac{p}{p-q}=[b_1, b_2, \ldots, b_k] :=b_1-
\cfrac{1}{b_2- \cfrac{1}{\ddots- \cfrac{1}{b_k}}}$$ be the Hirzebruch-Jung continued fraction expansion, where $b_i \geq 2$ for $1 \leq i \leq
k$. Note that the sequence $ b_1, b_2, \ldots, b_k$ is uniquely determined by the pair $(p,q)$.

\begin{Def} For any integer $k \geq 2$, a $k$-tuple of positive integers $(n_1, \ldots, n_k)$ is
called  admissible if each of the denominators in the continued
fraction $[n_1, \ldots, n_k]$ is positive. \end{Def}

\begin{Def} For any integer $k \geq 2$, let $\mathcal{Z}_{k} \subset
\mathbb{Z}^k$ denote the set of admissible $k$-tuples of positive
integers $\textbf{n}=(n_1, \ldots, n_k)$ such that $[n_1, \ldots,
n_k] =0$ and let $\mathcal{Z}_{1}=\{(0)\}$.  We set  $$\mathcal{Z}_{k}(\textstyle{\frac{p}{p-q}}) = \{ (n_1,
\ldots, n_k)\in\mathcal{Z}_{k}\,|\, 0 \leq n_i \leq b_i \;
\mbox{for} \; i=1, \ldots , k\}.$$ \end{Def}

For any integer $k \geq 2$, let $\mathcal{P}_{k+1}$ denote a convex polygon in the plane with $k+1$ vertices.  There is a simple identification of the set $\mathcal{Z}_{k}$ with the set $\mathcal{T} (\mathcal{P}_{k+1})$ of all triangulations of $\mathcal{P}_{k+1}$ due to Stevens \cite{s} as follows. Fix and label a distinguished vertex of $\mathcal{P}_{k+1}$ by  $V_\star$ and label the rest of the vertices as $V_1, \ldots, V_k$ traveling counterclockwise around $\mathcal{P}_{k+1}$. To each triangulation $\D \in \mathcal{T} (\mathcal{P}_{k+1})$, associate the $k$-tuple   $\textbf{n}=(n_1, \ldots, n_k) \in \mathcal{Z}_{k}$ so that $n_i$ is the number of triangles in $\D$ including the vertex $V_i$, which gives an explicit bijection from $\mathcal{T}(\mathcal{P}_{k+1})$ to $\mathcal{Z}_{k}$.

\begin{Def} For any $k \geq  2$, we denote the Stevens' bijection described  above as  $$\Phi_k \colon \mathcal{T} (\mathcal{P}_{k+1}) \to \mathcal{Z}_{k},$$ and  set $$\mathcal{T}^{p,q} (\mathcal{P}_{k+1}):= \Phi^{-1}_k (\mathcal{Z}_{k}(\textstyle{\frac{p}{p-q}})).$$\end{Def}

\begin{Rem} {\em Note that for any $k \geq  2$, $$|\mathcal{T}(\mathcal{P}_{k+1})| = |\mathcal{Z}_{k}| = \dfrac{1}{k} {2k-2 \choose k-1}, $$ which is nothing but the {\em Catalan number} $C_{k-1}$. } \end{Rem}

\begin{Def} \label{def: blowup} Let $s$ be an integer greater than or equal to $2$.  For any $1 \leq j \leq s-1$, the  blowup of an $s$-tuple $(n_1,  \ldots, n_s)$ of positive
integers at the $j$th term is the $(s+1)$-tuple
 $(n_1, \ldots,
n_{j-1}, n_{j}+1,1,n_{j+1}+1, n_{j+2}, \ldots,
n_s)$. We call such a blowup as an interior blowup. The exterior blowup of an $s$-tuple $ (n_1, \ldots,
n_s)$ of positive
integers is the $(s+1)$-tuple $(n_1, \ldots, n_{s-1},
n_s+1,1).$
We also say that $(0) \to (1,1)$ is the initial blowup. The inverse of a blowup is called a blowdown.
\end{Def}

It is well-known (see, for example, \cite[Lemma 2]{l}) that for any $\textbf{n} \in \mathcal{Z}_k$, there is a blowup sequence $$ (0) \to (1,1) \to \cdots \to \textbf{n}$$ starting with the initial blowup $(0) \to (1,1)$ and ending with $\textbf{n}$, although such a blowup sequence is not necessarily unique. This observation leads to the following definition of the height of $\textbf{n}$, which appeared  in \cite{bo}.

\begin{Def} \label{def: height} For $k \geq 2$, we say that
$\textbf{n} =(n_1,\ldots,n_k)\in\mathcal{Z}_{k}$ has  height $r$, and denote it  by $\hgt(\textbf{n})=r$, if $r$ is the minimal number of blowups
required to obtain $\textbf{n}$ from an $s$-tuple  $\textbf{u}_s = (1,\underbrace{2,\ldots,2}_{s-2},1)\in \mathbb{Z}^s$, for some $s \geq 2$. We also set $\textbf{u}_1=(0)$ and $\hgt(\textbf{u}_1)=0$.
\end{Def}

It follows that $\hgt(\textbf{u}_s)=0$ for all $s \geq 1$. \smallskip

\begin{Ex}{\em Consider the following blowup sequence  $$ (0)  \xrightarrow[\text{initial}]{\text{}} (1,1) \xrightarrow[\text{exterior}]{\text{}} (1,2,1) \xrightarrow[\text{interior}]{\text{}} (1,3,1,2)  \xrightarrow[\text{interior}]{\text{}} (1,3,2,1,3) $$ $$ \xrightarrow[\text{exterior}]{\text{}} (1,3,2,1,4,1) \xrightarrow[\text{interior}]{\text{}} (2,1,4,2,1,4,1) =\textbf{n} \in \mathcal{Z}_{7}.  $$ Note that there is a blowdown sequence $$ \textbf{n} = (2,\blue{1},4,2,1,4,1) \rightarrow (1,3,2,\blue{1},4,1)  \rightarrow (1,3,\blue{1},3,1) \rightarrow (1,2,2,1)= \textbf{u}_4, $$ obtained by blowing down at the leftmost {\em interior} $\blue{1}$ at each step, which shows that $\hgt(\textbf{n}) \leq 3$, and as a matter of fact $\hgt(\textbf{n}) = 3$ by Lemma~\ref{lem: height} below, an observation that was mentioned in \cite[page 1526]{bo}, without a proof. Note that there are two other blowdown sequences   $$ \textbf{n} = (2,1,4,2,\blue{1},4,1) \rightarrow (2,1,4,\blue{1},3,1)  \rightarrow (2,\blue{1},3,2,1)  \rightarrow (1,2,2,1)= \textbf{u}_4, $$ $$ \textbf{n} = (2,1,4,2,\blue{1},4,1) \rightarrow (2,\blue{1},4,1,3,1)  \rightarrow (1,3,\blue{1},3,1)  \rightarrow (1,2,2,1)= \textbf{u}_4, $$ obtained similarly but by making different choices. } \end{Ex}

\begin{Lem} \label{lem: height} By setting,  $|\textbf{n}|=n_1+\cdots+n_k$, we have $ \hgt(\textbf{n})=|\textbf{n}|-2(k-1)$ for any  $\textbf{n} \in\mathcal{Z}_{k}$.
\end{Lem}

\begin{proof} Write $g(\textbf{n})=|\textbf{n}|-2(k-1)$, for $\textbf{n}\in\mathcal{Z}_k$. It is easy to see that for $\textbf{n}\in\mathcal{Z}_k$, $g(\textbf{n})=0$ if and only if $\textbf{n}=\textbf{u}_k$. Since $0\leq g(\textbf{n}) - g(\textbf{n}')\leq 1$ whenever $\textbf{n}'$ is obtained by blowing down $\textbf{n}$, it follows that $\hgt(\textbf{n})\geq g(\textbf{n})$. We check that the inequality $\hgt(\textbf{n})\leq g(\textbf{n})$
also holds. To see this note that if $\textbf{n}\neq \textbf{u}_k$, then we can always perform an interior blowdown on $\textbf{n}$. Indeed,
suppose that for some $\textbf{n}\in\mathcal{Z}_k$ that is different from $\textbf{u}_k$, there are no interior $1$'s. Then blowing down
$\textbf{n}$, necessarily at an exterior $1$, would give a $(k-1)$-tuple $\textbf{n}'$ which again had no interior $1$'s. Repeatedly blowing down we
must eventually get $\textbf{u}_2$. Since each blowdown was at an exterior $1$, the original $k$-tuple $\textbf{n}$ must be
$\textbf{u}_k$, contrary to assumption. Thus blowing down $\textbf{n}$ at an interior $1$ $g(\textbf{n})$-times will
give $\textbf{u}_{k-g(\textbf{n})}$. It follows that we have $\hgt(\textbf{n})\leq g(\textbf{n})$ and hence  $\hgt(\textbf{n})=g(\textbf{n})$.
\end{proof}

\begin{Rem} {\em It follows from the proof of Lemma~\ref{lem: height}  that $ \hgt(\textbf{n})$ is the number of {\em interior} blowups in {\em any} blowup sequence $ (0) \to (1,1) \to \cdots \to \textbf{n}$.} \end{Rem}

\section{Diagonal flips along distinguished diagonals}\label{sec: flip}

In the  convex polygon $\mathcal{P}_{k+1}$, with the fixed distinguished vertex $V_\star$ and the rest of the vertices $V_1, \ldots, V_k$ labelled counterclockwise as in Section~\ref{sec: basic},  there are exactly $k-2$ {\em distinguished} diagonals $d_1, \ldots, d_{k-2}$, defined so that for each $1 \leq i \leq k-2$, the diagonal  $d_i$ connects $V_\star$ to the vertex $V_{i+1}$.

\begin{Def}  For any integer $k \geq 3$, the  triangulation $\D_\star \in T (\mathcal{P}_{k+1})$ which is obtained by using precisely the set $\{d_1, \ldots , d_{k-2}\}$ of all distinguished diagonals is called the initial triangulation.
\end{Def}

\begin{figure}[ht]  \relabelbox \small {\epsfxsize=6in
\centerline{\epsfbox{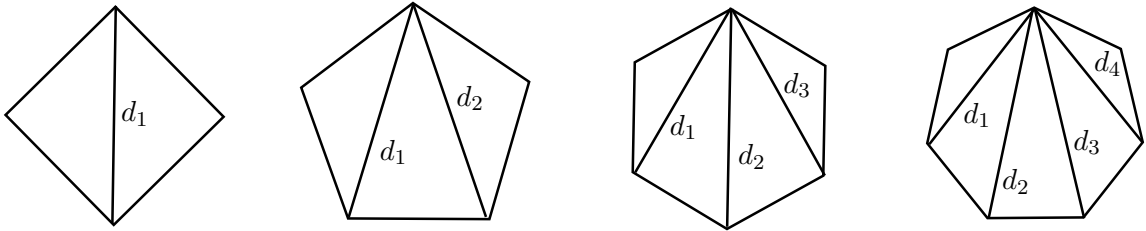}}}
\relabel{1}{$d_1$} \relabel{2}{$d_1$} \relabel{3}{$d_2$} \relabel{4}{$d_1$} \relabel{5}{$d_2$} \relabel{6}{$d_3$} \relabel{7}{$d_1$} \relabel{8}{$d_2$} \relabel{9}{$d_3$} \relabel{10}{$d_4$}
 \endrelabelbox
\caption{Initial triangulations for $k=3,4,5,6,$ respectively. }\label{fig: initialtri} \end{figure}

In Figure~\ref{fig: initialtri}, we depicted the initial triangulations for $k=3,4,5,6$. Note that if $d$ is {\em any} diagonal which appears in any triangulation  $\D \in \mathcal{T} (\mathcal{P}_{k+1})$, then the union of the  two triangles on either side of $d$ makes up a quadrilateral which is bisected by $d$ into two triangles of $\D$.

\begin{Def} \label{def: flip} Suppose that $\D \in \mathcal{T} (\mathcal{P}_{k+1})$ is a triangulation which includes a distinguished diagonal $d_i$ for some $1 \leq i \leq k-2$. A diagonal flip of $\D$ along $d_i$ is a transformation of $\D$ into another triangulation $\widetilde{\D} \in \mathcal{T} (\mathcal{P}_{k+1})$ where $d_i$ is replaced by the unique non-distinguished diagonal $\widetilde{d}_{i}$ of the  unique quadrilateral which is bisected by $d_i$ into two triangles of $\D$.
\end{Def}

In Figure~\ref{fig: diagonalflip}, for example, we depicted a sequence of diagonal flips along distinguished diagonals, starting from the initial triangulation of the heptagon.

\begin{figure}[ht]  \relabelbox \small {\epsfxsize=6in
\centerline{\epsfbox{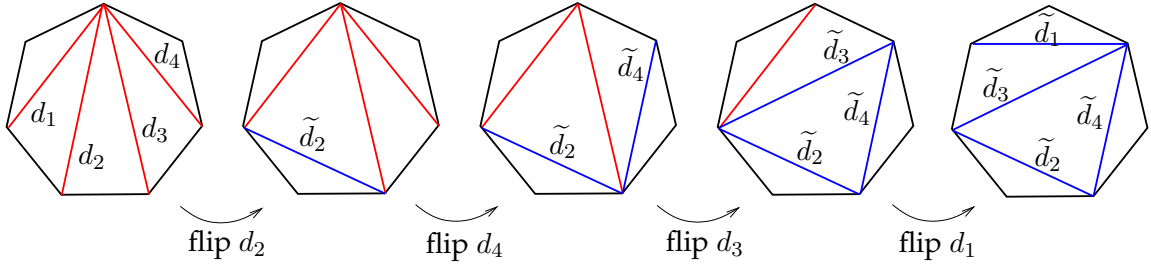}}}
 \relabel{7}{$d_1$} \relabel{8}{$d_2$} \relabel{1}{$\widetilde{d}_1$} \relabel{2}{$\widetilde{d}_2$} \relabel{3}{$\widetilde{d}_3$} \relabel{4}{$\widetilde{d}_4$}
\relabel{9}{$d_3$} \relabel{10}{$d_4$} \relabel{d}{flip $d_1$} \relabel{a}{flip $d_2$} \relabel{c}{flip $d_3$} \relabel{b}{flip $d_4$}  \relabel{e}{$\widetilde{d}_2$} \relabel{f}{$\widetilde{d}_2$} \relabel{g}{$\widetilde{d}_2$} \relabel{h}{$\widetilde{d}_4$} \relabel{i}{$\widetilde{d}_4$} \relabel{j}{$\widetilde{d}_3$}
 \endrelabelbox
\caption{A sequence of diagonal flips. } \label{fig: diagonalflip} \end{figure}

\begin{Rem} \label{rem: dual} Note that the non-distinguished diagonal $\widetilde{d}_{i}$ in Definition~\ref{def: flip}, depends on the quadrilateral which is determined  by specifying  its three non-distinguished vertices. We will refer to  $\widetilde{d}_{i}$ as the dual of the distinguished diagonal $d_i$ in that quadrilateral.  In other words, the dual diagonal $\widetilde{d}_{i}$ is the "image" of the distinguished diagonal $d_i$ under the diagonal flip move.  \end{Rem}

In the following, for each integer $k \geq 3$, we will describe the graded, directed, rooted, connected graph $\mathcal{G}_k$ which organizes the triangulations of the convex polygon $\mathcal{P}_{k+1}$ with respect to their {\em heights}.

\begin{Def} \label{def: ht} The height of a triangulation $\D \in \mathcal{T} (\mathcal{P}_{k+1})$  is defined as the height of $\Phi_k (\D) \in \mathcal{Z}_{k}$ under the Stevens' bijection $\Phi_k$. \end{Def}

Next we show that each diagonal flip along a distinguished diagonal increases the height of a given triangulation by one.

\begin{Lem} \label{lem: hgt}  Suppose that $\D \in  \mathcal{T}  (\mathcal{P}_{k+1})$ is a triangulation which includes a distinguished diagonal $d$. If $\widetilde{\D} \in  \mathcal{T}  (\mathcal{P}_{k+1})$ is the triangulation obtained from $\D$ by the diagonal flip along $d$, then  $\hgt(\widetilde{\D}) = \hgt(\D)+1$.
\end{Lem}

\begin{proof}  By  Definition~\ref{def: flip}, the diagonal flip along the distinguished diagonal $d$ occurs in a quadrilateral which has one distinguished vertex $V_\star$ and three other vertices, say $V_r, V_s, V_t$, ordered counterclockwise, where the distinguished diagonal $d$ that connects the vertices  $V_\star$ and $V_s$ is exchanged with the dual diagonal $\widetilde{d}$ that connects the vertices $V_r$ and $V_t$. As a result of this exchange, the number of triangles including the vertex $V_s$ decreases by $1$, but   the number of triangles including each of the two remaining non-distinguished vertices $V_r$ and $V_t$ of the quadrilateral increases by $1$.  Since the number of triangles including each vertex of the polygon $\mathcal{P}_{k+1}$, other than $V_r, V_s$ and $V_t$ remains the same,  it follows that $\hgt(\widetilde{\D}) = \hgt(\D)+1$, by Lemma~\ref{lem: height}.
\end{proof}

\begin{Def} \label{def: inverse} For any $\textbf{n} \in \mathcal{Z}_k$, we set $\D (\textbf{n}) = \Phi_k^{-1} (\textbf{n}) \in \mathcal{T} (\mathcal{P}_{k+1})$, where $$\Phi_k \colon   \mathcal{T} (\mathcal{P}_{k+1}) \to  \mathcal{Z}_k$$ is the Stevens' bijection.   \end{Def}

It follows by Definition~\ref{def: ht} that  $\hgt (\D (\textbf{n})) = \hgt(\textbf{n})$.

\begin{Prop} \label{prop: graph} For any integer $k  \geq 3$, there is a graded, directed, rooted, connected graph $\mathcal{G}_k$ such that
\begin{enumerate}
\item{there is a bijection $\psi_k$ from  the set  $\mathcal{T}(\mathcal{P}_{k+1})$ of all triangulation of the convex polygon $\mathcal{P}_{k+1}$ with $k+1$ vertices,  to the set of vertices of $\mathcal{G}_k$,  }
\item{the root of $\mathcal{G}_k$ is the image $\psi_k (\D _\star) $ of the initial triangulation  $\D _\star = \D (\textbf{u}_k)$, }
\item{if the triangulation $\widetilde{\D} \in \mathcal{T}(\mathcal{P}_{k+1})$ is obtained from the triangulation $\D \in \mathcal{T}(\mathcal{P}_{k+1})$ by a single diagonal flip along a distinguished diagonal,  then there is a directed edge from  the vertex $\psi_k (\D)$ to the vertex $\psi_k (\widetilde{\D})$, and }
\item{each vertex $\psi_k (\D)$ is graded by the height of $\D$ and the grading increases by one along each directed edge.}
\end{enumerate}
\end{Prop}

\begin{proof} Fix any integer $k  \geq 3$. We take the set $\mathcal{T}(\mathcal{P}_{k+1})$  of triangulations of the convex polygon $\mathcal{P}_{k+1}$ as the vertices of our graph $\mathcal{G}_k$, which implicitly defines the bijection $\psi_k$ in Proposition~\ref{prop: graph}. In the following we suppress $\psi_k$ from the notation.  To construct the graph $\mathcal{G}_k$, we organize the triangulations in $\mathcal{T}(\mathcal{P}_{k+1})$ with respect to their heights. We define the root of $\mathcal{G}_k$ as the initial triangulation $ \D _\star = \D (\textbf{u}_k) \in  \mathcal{T}  (\mathcal{P}_{k+1})$, which is the only triangulation of height zero, by Lemma~\ref{lem: height}.

Right below the root vertex $\D (\textbf{u}_k)$, we place vertices in the first row, corresponding to height $1$ triangulations of $\mathcal{P}_{k+1}$, each of which is obtained from  $\D (\textbf{u}_k)$ by a single diagonal flip along a distinguished diagonal. Since there are $k-2$ distinguished diagonals $\{d_1, \ldots, d_{k-2}\}$ of $\mathcal{P}_{k+1}$, we have $k-2$ height $1$ triangulations which are naturally ordered from left to right according to which distinguished diagonal we flip. Moreover, we insert an edge that connects the root vertex to each of the height $1$ triangulations. Hence the root vertex has no incoming edges, by definition, and is connected to $k-2$ distinct vertices by outgoing edges denoted $e_1, \ldots, e_{k-2}$, respectively, so that the end vertex of $e_i$ is obtained from $\D (\textbf{u}_k)$ by the diagonal flip along the distinguished diagonal $d_i$. See Figure~\ref{fig: hexagon1} for the case $k=5$.

\begin{figure}[ht] \small {\epsfxsize=5.5in
\centerline{\epsfbox{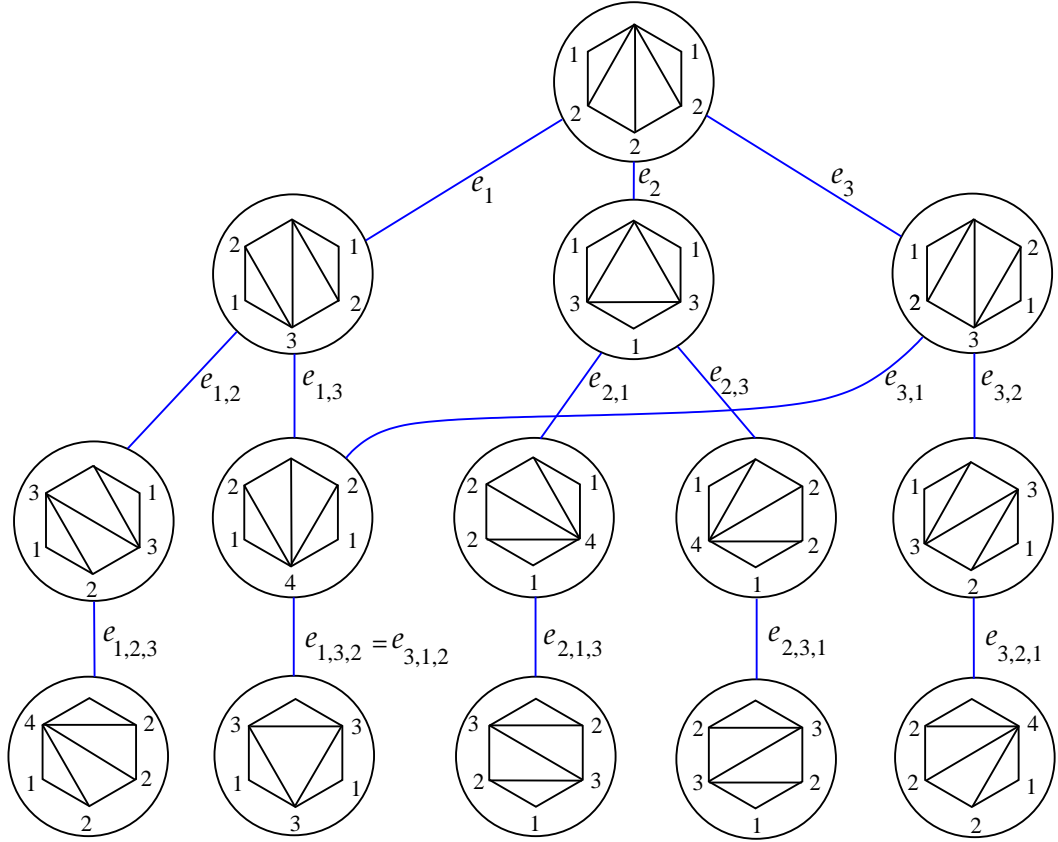}}}
\caption{The graded, directed, rooted, connected graph  $\mathcal{G}_5$ of the triangulations of the hexagon. } \label{fig: hexagon1} \end{figure}

Next we place the triangulations (i.e. vertices) of height $2$ in a row right below the row of vertices of height $1$, and insert the connecting edges between height $1$ and height $2$ vertices as follows. Note that each height $2$ triangulation is obtained from some height $1$ triangulation by a single diagonal flip. Now consider the left-most vertex in the first row, and apply single diagonal flip along each of the remaining distinguished diagonals, in the order of increasing indices of the diagonals. Then move on to the next vertex in the first row, repeat the same process and place the new vertices in the second row right next to the already existing vertices.  It is clear that one can apply the same process for each of the vertices in the first row, with the {\em caveat} that some height $2$ triangulation may be obtained from two distinct height $1$ triangulations. To avoid repetitions, we employ the following rule: if a height $2$ triangulation already appears in the second row, we do not insert a new vertex if the same triangulation can be obtained from another height $1$ triangulation. For example, the triangulation $\D ((2,1,4,1,2))$ of height $2$ in Figure~\ref{fig: hexagon1} can be obtained either from the triangulation $\D((2,1,3,2,1))$ or the triangulation $\D((1,2,3,1,2))$.

Now we explain how to insert edges between height $1$ and height $2$ vertices. If a height $2$ triangulation is obtained by a diagonal flip along a distinguished diagonal $d_j$ from a height $1$ triangulation which is the end point of some edge $e_i$ (where we necessarily have $i \neq j$), then we insert an edge, denoted $e_{i,j}$, to connect the height $1$ triangulation to the height $2$ triangulation.

It should be clear that this procedure can be iterated until there are no more distinguished diagonals to be flipped, so that in the $r$th row, we have all the height $r$ triangulations of  $\mathcal{P}_{k+1}$, without any repetitions. We {\em orient} every edge so that the height of the end point is one higher than the height of the source. Moreover, each vertex of height $r \geq 1$ in $\mathcal{G}_k$ has {\em at least} one incoming edge and exactly $k-r-2$ outgoing edges. Note that some edges might have multiple names. For example the edge $e_{1,3,2}$ is the same as the edge $e_{3,1,2}$ in Figure~\ref{fig: hexagon1}.

To finish the proof, we need to show that every triangulation $\D \in  \mathcal{T} (\mathcal{P}_{k+1})$ appears once in $\mathcal{G}_k$ and that $\mathcal{G}_k$ is connected.  This follows from  Lemma~\ref{lem: once} below. \end{proof}

\begin{Lem} \label{lem: once} Let $\D(\textbf{n})$ be a triangulation of $\mathcal{P}_{k+1}$ and suppose that $\hgt(\textbf{n})=s$. Then there is a one-to-one correspondence between blowdown sequences $$\textbf{n}=\textbf{n}^0 \to \textbf{n}^1 \to \cdots \to \textbf{n}^s=\textbf{u}_{k-s},$$ where each blowdown is at an interior $1$, and paths
$$ \D(\textbf{u}_k)=\D(\textbf{n}_0),\D(\textbf{n}_1),\ldots,\D(\textbf{n}_s)=\D(\textbf{n}) $$
in $\mathcal{G}_k$
starting at the root vertex $\D(\textbf{u}_k)$ and ending at $\D(\textbf{n})$. \end{Lem}

\begin{proof}   The first blowdown
$\textbf{n}^0\to\textbf{n}^1$ at an interior $1$ of $\textbf{n}^0$  corresponds geometrically to peeling off from $\D(\textbf{n})$ a triangle, called  $\tau_1$,
that has two edges along the boundary of $\mathcal{P}_{k+1}$ meeting at a vertex of $\mathcal{P}_{k+1}$, which corresponds to the interior $1$ that we blow down.   Let $d_{i_1}$ be the distinguished diagonal of $\mathcal{P}_{k+1}$ connecting this vertex to the distinguished vertex of $\mathcal{P}_{k+1}$. Note that there is a unique quadrilateral in $\mathcal{P}_{k+1}$, whose non-distinguished vertices are precisely the vertices of $\tau_1$, and the interior edge of $\tau_1$ is dual (see Remark~\ref{rem: dual}) to  $d_{i_1}$ in that quadrilateral, so that the interior edge of  $\tau_1$ is denoted by  $\widetilde{d}_{i_1}$.

Since we peeled off $\tau_1$ from $\mathcal{P}_{k+1}$ corresponding to the first blowdown $\textbf{n}^0\to\textbf{n}^1$, the remaining polygon can be identified with $\mathcal{P}_{k}$, which is embedded in $\mathcal{P}_{k+1}$.  The second blowdown
$\textbf{n}^1 \to \textbf{n}^2$ at an interior $1$ of $\textbf{n}^1$  corresponds geometrically to peeling off from $\D(\textbf{n}^1)$ a triangle, called  $\tau_2$,
that has two edges along the boundary of $\mathcal{P}_{k}$ meeting at a vertex of $\mathcal{P}_{k}$, which corresponds to the interior $1$ that we blow down. The distinguished diagonal of $\mathcal{P}_{k}$ connecting this vertex to the distinguished vertex of $\mathcal{P}_{k}$ is also a distinguished diagonal of   $\mathcal{P}_{k+1}$ by the embedding of $\mathcal{P}_{k}$ into $\mathcal{P}_{k+1}$. We label this distinguished diagonal of $\mathcal{P}_{k+1}$ as $d_{i_2}$.  Moreover,  there is a unique quadrilateral in $\mathcal{P}_{k} \subset \mathcal{P}_{k+1}$, whose non-distinguished vertices are precisely the vertices of $\tau_2$, where the interior edge of $\tau_2$ is dual to  $d_{i_2}$ in that quadrilateral, so that the interior edge of  $\tau_2$ is denoted by  $\widetilde{d}_{i_2}$.

Continuing in this way, until we arrive at  $\textbf{n}^s=\textbf{u}_{k-s}$, we obtain a sequence of triangles $\tau_1,\tau_2,\ldots,\tau_s$ in $\D(\textbf{n})$ and
a sequence of diagonals $d_{i_1},d_{i_2},\dots,d_{i_s}$ of $\mathcal{P}_{k+1}$. It follows by our construction that flipping the diagonals of the initial
triangulation $\D(\textbf{u}_k)\in\mathcal{T}(\mathcal{P}_{k+1})$ in the order $d_{i_1},d_{i_2},\dots,d_{i_s}$
gives precisely the triangulation $\D(\textbf{n})$.

It is easy to see that this process can be reversed.
Namely, any path starting from the root vertex of $\mathcal{G}_k$ and ending at a vertex $\D(\textbf{n})$ is
uniquely specified by a sequence of distinguished diagonals of $\mathcal{P}_{k+1}$, and this sequence of diagonals
via the canonically associated sequence of triangles specifies a unique blowdown sequence starting at $\textbf{n}$
and ending at $\textbf{u}_{k-s}$.
\end{proof}

\begin{Ex} \label{ex: down} {\em We illustrate the proof of Lemma~\ref{lem: once} for $\textbf{n} = (3,2,1,4,2,1,4) \in \mathcal{Z}_7$. Let $\D(\textbf{n}) \in \mathcal{G}_7$ be the corresponding triangulation of the octagon depicted in  Figure~\ref{fig: octagon2}. We observe that $\hgt(\textbf{n})=5$ by Lemma~\ref{lem: height}  and take the blowdown sequence
$$ \textbf{n}=(3,2,1,4,2,1,4) \to \textbf{n}^1 = (3,1,3,2,1,4) \to \textbf{n}^2 = (2,2,2,1,4)   \to $$ $$ \textbf{n}^3 = (2,2,1,3) \to   \textbf{n}^4 = (2,1,2) \to \textbf{n}^5=(1,1)=\textbf{u}_{2},$$ which we depicted geometrically in the top row of  Figure~\ref{fig: octagon}, by peeling off the triangles $\tau_1, \tau_2, \ldots,  \tau_5$ in order. The corresponding path  $$ \D(\textbf{n}_0)=\D(\textbf{u}_7)=\D((1,2,2,2,2,2,1)),\D(\textbf{n}_1) = \D((1,3,1,3,2,2,1)),$$ $$\D(\textbf{n}_2) = \D((2,2,1,4,2,2,1)), \D(\textbf{n}_3) = \D((2,2,1,4,3,1,2)),$$  $$ \D(\textbf{n}_4) = \D((2,2,1,5,2,1,3)), \D(\textbf{n}_5)=\D(\textbf{n}) $$ of vertices
of  $\mathcal{G}_7$ is depicted in the bottom row of  Figure~\ref{fig: octagon}.

\begin{figure}[ht]  \small {\epsfxsize=1.4in
\centerline{\epsfbox{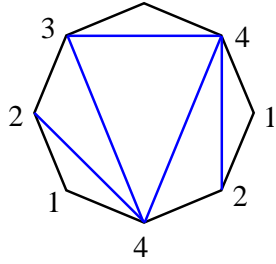}}}
\caption{The triangulation $\D((3,2,1,4,2,1,4))$ of the octagon.  }\label{fig: octagon2} \end{figure}

\begin{figure}[ht]  \relabelbox \small {\epsfxsize=6in
\centerline{\epsfbox{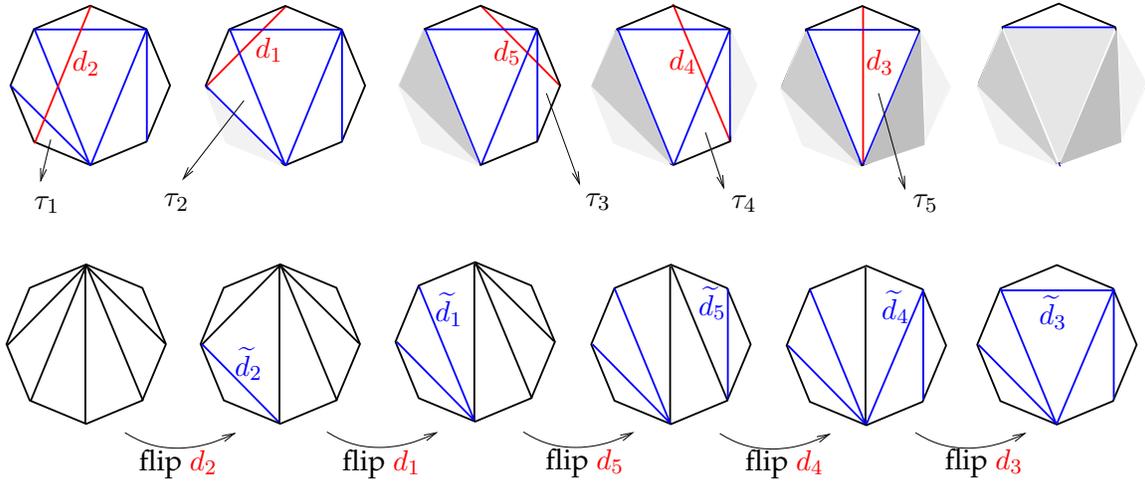}}}
\relabel{1}{\red{$d_2$}} \relabel{2}{\red{$d_1$}} \relabel{3}{\red{$d_5$}} \relabel{4}{\red{$d_4$}} \relabel{5}{\red{$d_3$}} \relabel{a}{$\tau_1$} \relabel{b}{$\tau_2$} \relabel{c}{$\tau_3$} \relabel{d}{$\tau_4$} \relabel{e}{$\tau_5$} \relabel{f}{\blue{$\widetilde{d}_2$}} \relabel{g}{\blue{$\widetilde{d}_1$}} \relabel{h}{\blue{$\widetilde{d}_5$}} \relabel{i}{\blue{$\widetilde{d}_4$}} \relabel{j}{\blue{$\widetilde{d}_3$}} \relabel{6}{flip \red{$d_2$}} \relabel{7}{flip \red{$d_1$}} \relabel{8}{flip \red{$d_5$}} \relabel{9}{flip \red{$d_4$}} \relabel{10}{flip \red{$d_3$}}
 \endrelabelbox
\caption{Peeling off triangles and flipping diagonals.}\label{fig: octagon} \end{figure}

Note that in this example, there are six distinct blowdown sequences starting from $ \textbf{n}=(3,2,1,4,2,1,4)$ and ending with $\textbf{u}_{2}$, which is equivalent to the fact that there are six distinct paths in the graph $\mathcal{G}_7$ from the root vertex $\D(\textbf{u}_7)$ to the vertex $\D(\textbf{n})$.}
\end{Ex}

\section{Flipping contiguously and Riemenschneider's point diagrams}

\begin{Def} Suppose that $p > q \geq 1$ are coprime integers and let
$\textstyle{\frac{p}{p-q}}=[b_1, b_2, \ldots, b_k]$ be the Hirzebruch-Jung continued fraction expansion, where $b_i \geq 2$ for $1 \leq i \leq
k$. We set $$\hj(\textstyle{\frac{p}{p-q}})= (b_1, b_2, \ldots, b_k).$$
\end{Def}

Note that the $k$-tuple $(b_1, b_2, \ldots, b_k)$ of integers  is uniquely determined by the pair $(p,q)$. Similarly, if $\textstyle{\frac{p}{q}}=[a_1, a_2, \ldots, a_r]$, then we set $\hj(\textstyle{\frac{p}{q}})= (a_1, a_2, \ldots, a_r).$ There is a duality between $\hj(\textstyle{\frac{p}{p-q}})= (b_1, b_2, \ldots, b_k),$ and $\hj(\textstyle{\frac{p}{q}})= (a_1, a_2, \ldots, a_r)$
obtained by using the  {\em Riemenschneider's point diagram method} \cite{r}: place in the $i$th row $b_i - 1$ dots, the first one under the last one of the $(i- 1)$st row; then column $j$ contains $a_j - 1$ dots. Using this method, one can compute $\hj(\textstyle{\frac{p}{p-q}})$ from $\hj(\textstyle{\frac{p}{q}})$  and vice-versa. See Figure~\ref{fig: point}, for an example of the Riemenschneider's point diagram method.

\begin{Def} Suppose that $$\D(\textbf{n}_0) \xrightarrow[\text{flip}]{\text{$d_{i_1}$}} \D(\textbf{n}_1) \xrightarrow[\text{flip}]{\text{$d_{i_2}$}} \D(\textbf{n}_2) \rightarrow \cdots \rightarrow \D(\textbf{n}_{r-1})  \xrightarrow[\text{flip}]{\text{$d_{i_r}$}} \D(\textbf{n}_r)$$ is a
path in $\mathcal{G}_k$, where $d_{i_1},d_{i_2},\ldots,d_{i_r}$ is the corresponding sequence of distinguished diagonals of $\mathcal{P}_{k+1}$ that are flipped along the edges of this path. We say that $d_{i_1},d_{i_2},\ldots,d_{i_r}$ is a contiguous sequence of distinguished diagonals in the triangulation $\D(\textbf{n}_0)$ if any successive pair $d_{i_{j+1}}, d_{i_{j+2}}$ bound a triangle in $\D(\textbf{n}_{j})$ for $ 0 \leq j \leq r-2$. We also say that $\D(\textbf{n}_0)$, $\D(\textbf{n}_1)$, $\ldots$, $\D(\textbf{n}_r)$ is a contiguous path in $\mathcal{G}_k$.

\end{Def}

\begin{Ex} \label{ex: cont} {\em In Figure~\ref{fig: ex}, we depicted a {\em contiguous} path in $\mathcal{G}_{9}$, starting from the initial triangulation $\D(\textbf{u}_9)$ and ending with $\D(2,2,2,4,2,1,3,2,5)$ obtained by flipping along the {\em contiguous} sequence $d_5$, $d_4$, $d_6$, $d_7$, $d_3$, $d_2$, $d_1$  of distinguished diagonals in the initial triangulation of the decagon.

\begin{figure}[ht]\small {\epsfxsize=6in
\centerline{\epsfbox{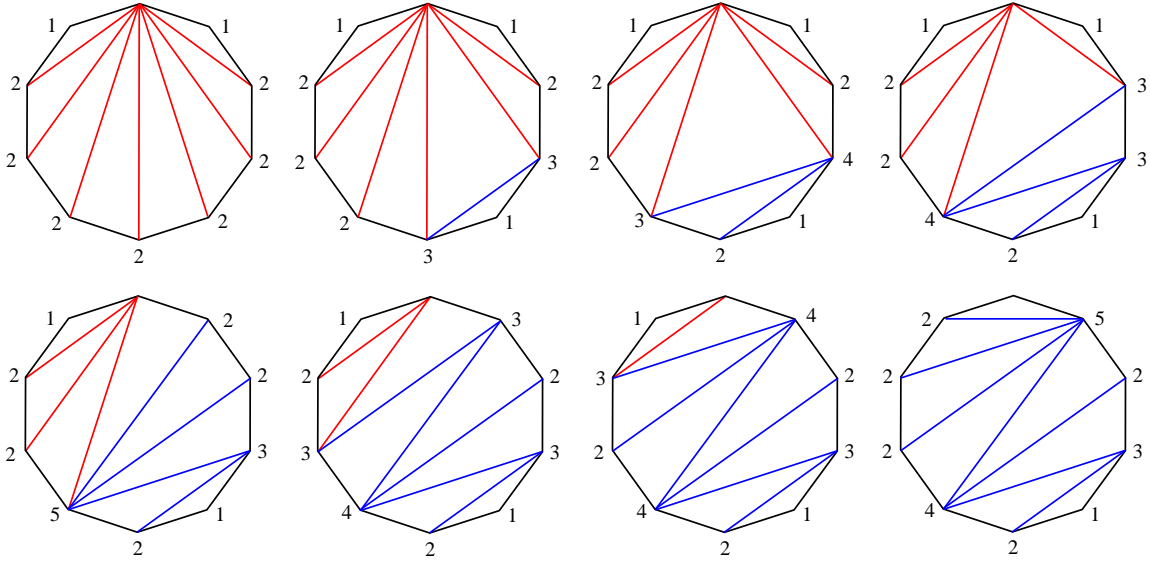}}}
\caption{Flips along a contiguous sequence  $d_5$, $d_4$, $d_6$, $d_7$, $d_3$, $d_2$, $d_1$  of distinguished diagonals in the initial triangulation of the decagon.} \label{fig: ex} \end{figure}
}
\end{Ex}

\begin{Rem} {\em In Example~\ref{ex: down}, the sequence $d_2$, $d_1$, $d_5$, $d_4$, $d_3$ of distinguished diagonals is not contiguous in $\D(\textbf{u}_7)$, since the pair $d_1$, $d_5$ do not bound a triangle in $\D(\textbf{n}_1)$. }
\end{Rem}

\begin{Lem} \label{lem: cont} Fix an integer $k \geq 3$ and let  $\textbf{n} \in \mathcal{Z}_k$ so that $\hgt  (\textbf{n} ) \geq 1$. Then the following are equivalent:
 \begin{enumerate}
 \item{there is a contiguous path from $\D(\textbf{u}_k) $ to $\D(\textbf{n})$ in $\mathcal{G}_k$   }
 \item{$\dpt (\textbf{n})= 1$ }
  \item{ there is a unique path from $\D(\textbf{u}_k) $ to $\D(\textbf{n})$ in $\mathcal{G}_k$  }
  \end{enumerate}
\end{Lem}

\begin{proof}
Suppose that  $\textbf{n} \in \mathcal{Z}_k$ such that $\hgt  (\textbf{n} ) =s \geq 1$.

\smallskip

$(1) \Rightarrow (2)$ Suppose that there is a contiguous path $$\D(\textbf{u}_k)=\D(\textbf{n}_0) \xrightarrow[\text{flip}]{\text{$d_{i_1}$}} \D(\textbf{n}_1) \xrightarrow[\text{flip}]{\text{$d_{i_2}$}} \D(\textbf{n}_2) \rightarrow \cdots \rightarrow \D(\textbf{n}_{s-1})  \xrightarrow[\text{flip}]{\text{$d_{i_{s}}$}} \D(\textbf{n}_{s})=\D(\textbf{n})$$ in $\mathcal{G}_k$, where $d_{i_1}, \ldots,  d_{i_{s}}$ is a contiguous sequence of distinguished diagonals  of  $\D(\textbf{u}_k)$.     We would like to show that $\dpt (\textbf{n})= 1$. First of all, in the triangulation $\D(\textbf{n}_1)$, obtained from $\D(\textbf{u}_k)$ by the diagonal flip along $d_{i_1}$, there is a unique triangle $\tau_1$ whose only interior edge is  $\widetilde{d}_{i_1}$. It follows that $\dpt (\textbf{n}_1)= 1$. Now if we peel away  $\tau_1$ from $\D(\textbf{n}_1)$, the result is identical to the initial triangulation $\D(\textbf{u}_{k-1})$ of $\mathcal{P}_{k}$ and $d_{i_2}$, $\ldots$, $d_{i_{s}}$ is a contiguous sequence of distinguished diagonals of $\D(\textbf{u}_{k-1})$. By a straightforward inductive argument and pasting the triangle $\tau_1$ back we see that $\dpt (\textbf{n})= 1$. Namely, $\textbf{n}$  has exactly one {\em interior} component that is equal to $1$.

$(2) \Rightarrow (3)$ Suppose that $\dpt (\textbf{n})= 1$, which, by definition,  means that $\textbf{n}$ has exactly one interior component that is equal to $1$. By blowing down successively the unique interior $1$ at each step, we get a blowdown sequence  $$\textbf{n}=\textbf{n}^0 \to \textbf{n}^1 \to \cdots \to \textbf{n}^{s}=\textbf{u}_{k-s}.$$  Then according to Lemma~\ref{lem: once}, there is   a unique path
$$\D(\textbf{u}_k)=\D(\textbf{n}_0), \D(\textbf{n}_1), \ldots,  \D(\textbf{n}_{s})=\D(\textbf{n})$$ in $\mathcal{G}_k$,
starting at the root vertex $\D(\textbf{u}_k)$ and ending at $\D(\textbf{n})$.

$(3) \Rightarrow (1)$ Suppose that there is a unique path
$$\D(\textbf{u}_k)=\D(\textbf{n}_0) \xrightarrow[\text{flip}]{\text{$d_{i_1}$}} \D(\textbf{n}_1) \xrightarrow[\text{flip}]{\text{$d_{i_2}$}} \D(\textbf{n}_2) \rightarrow \cdots \rightarrow \D(\textbf{n}_{s-1})  \xrightarrow[\text{flip}]{\text{$d_{i_{s}}$}} \D(\textbf{n}_{s})=\D(\textbf{n})$$ in $\mathcal{G}_k$, where $d_{i_1}, \ldots,  d_{i_{s}}$ is a sequence of distinguished diagonals  of  $\D(\textbf{u}_k)$. Then  we claim that $d_{i_1}, \ldots, d_{i_{s}}$ must be  a {\em contiguous} sequence of distinguished diagonals in $\D(\textbf{n})$.  Indeed if $d_{i_j}$ and $d_{i_{j+1}}$ were not adjacent for some $j$, then
we could find an alternate path from $\D (\textbf{u}_k)$ to $\D (\textbf{n})$ in $\mathcal{G}_k$ by interchanging the order in which we
flip $d_{i_j}$ and $d_{i_{j+1}}$, contradicting the uniqueness of the path between $\D (\textbf{u}_k)$ to $\D (\textbf{n})$ in $\mathcal{G}_k$.
\end{proof}

\begin{Lem} \label{lem: dep} Fix an integer $k \geq 3$ and let  $\textbf{n}, \textbf{n}' \in \mathcal{Z}_k$. Suppose that there exists a unique path from $\D(\textbf{n})$ to $\D(\textbf{n}')$ in $\mathcal{G}_k$. Then $\dpt (\textbf{n}') - \dpt (\textbf{n}) \leq  1$.
\end{Lem}

\begin{proof}  Let  $\textbf{n}, \textbf{n}' \in \mathcal{Z}_k$ and suppose that there exists a unique path $$\D(\textbf{n})=\D(\textbf{n}_0) \xrightarrow[\text{flip}]{\text{$d_{i_1}$}} \D(\textbf{n}_1) \xrightarrow[\text{flip}]{\text{$d_{i_2}$}} \D(\textbf{n}_2) \rightarrow \cdots \rightarrow \D(\textbf{n}_{t-1})  \xrightarrow[\text{flip}]{\text{$d_{i_{t}}$}} \D(\textbf{n}_{t})=\D(\textbf{n}')$$ in $\mathcal{G}_k$. Then this path is contiguous by the same argument given in the proof of $(3) \Rightarrow (1)$ in Lemma~\ref{lem: cont}. Note that by the contiguity of the above path, we have   $\dpt (\textbf{n}_j) = \dpt (\textbf{n}_1)$, for all $2 \leq j \leq t$. It follows that $\dpt (\textbf{n}') = \dpt (\textbf{n})+1$ if $\dpt (\textbf{n}_1) = \dpt (\textbf{n}_0)+1$, and otherwise $\dpt (\textbf{n}') = \dpt (\textbf{n})$.
\end{proof}

\begin{Def} For any integer $s\geq 1$,  let $\mathcal{O}^s$ denote the set of all $s$-tuples that can be obtained from the $1$-tuple $(4)$ and applying the following iterations successively: \begin{enumerate}[label=(\alph*)]
\item {Insert $2$ as the first component and increase the last component by $1$, or}
\item {Insert $2$ as the last component  and increase the first component by $1$.} \end{enumerate} We set $\mathcal{O} = \bigcup_{s \geq 1} \mathcal{O}^s$.
\end{Def}

\begin{Lem} \label{lem: plumbing} Fix an integer $k \geq 3$ and let  $\textbf{n} \in \mathcal{Z}_k$ so that $\hgt  (\textbf{n} ) = k-2$ and $\dpt  (\textbf{n} ) = 1$.  Suppose that
$$\D(\textbf{u}_k)=\D(\textbf{n}_0) \xrightarrow[\text{flip}]{\text{$d_{i_1}$}} \D(\textbf{n}_1) \xrightarrow[\text{flip}]{\text{$d_{i_2}$}} \D(\textbf{n}_2) \rightarrow \cdots \rightarrow \D(\textbf{n}_{k-3})  \xrightarrow[\text{flip}]{\text{$d_{i_{k-2}}$}} \D(\textbf{n}_{k-2})=\D(\textbf{n})$$ is the unique contiguous path in $\mathcal{G}_k$ (as described in Lemma~\ref{lem: cont}), where $d_{i_1}, \ldots, d_{i_{k-2}}$ is a contiguous  sequence of distinguished diagonals in $\D(\textbf{u}_k)$. Let $\textbf{m}$ denote the $k$-tuple having $1$ in the position that $\textbf{n}$ has a $1$ and $0$ elsewhere, and let $1\leq q < p $ be the unique coprime integers such that $\hj(\textstyle{\frac{p}{p-q}})=\textbf{n}+ \textbf{m}$. Then $\hj(\textstyle{\frac{p}{q}})$ belongs to $\mathcal{O}^{k-2}$ and moreover, it can be described by starting from the $1$-tuple $(4)$ and applying $k-3$ iterations according to the following rule:  \begin{enumerate}[label=(\alph*)]
\item {If $i_j < i_{j+1}$, then insert $2$ as the first component and increase the last component by $1$, and }
\item {If $i_j > i_{j+1}$, then insert $2$ as the last component and increase the first component  by $1$.}
\end{enumerate} \end{Lem}

\begin{Rem} \label{rem: int} {\em Because of the assumption  $\hgt  (\textbf{n} ) = \hgt  (\D(\textbf{n} )) = k-2$, the triangulation $\D(\textbf{n})$ is obtained from $\textbf{u}_k$ by flipping all the distinguished diagonals in $\textbf{u}_k$ in some order. It follows that if any component of the $k$-tuple $\textbf{n}$ is equal to $1$, it must be an interior component.  So, the condition $\dpt (\textbf{n})= 1$ is equivalent to the condition that $\textbf{n}$ has exactly one component that is equal to $1$, which is in the interior of $\textbf{n}$.}
\end{Rem}

    \begin{proof}[Proof of  Lemma~\ref{lem: plumbing}.] Let  $\textbf{n} \in \mathcal{Z}_k$ so that $\hgt  (\textbf{n} ) = k-2$ and $\dpt  (\textbf{n} ) = 1$. By  Lemma~\ref{lem: cont}, there is a contiguous path $$\D(\textbf{u}_k)=\D(\textbf{n}_0) \xrightarrow[\text{flip}]{\text{$d_{i_1}$}} \D(\textbf{n}_1) \xrightarrow[\text{flip}]{\text{$d_{i_2}$}} \D(\textbf{n}_2) \rightarrow \cdots \rightarrow \D(\textbf{n}_{k-3})  \xrightarrow[\text{flip}]{\text{$d_{i_{k-2}}$}} \D(\textbf{n}_{k-2})=\D(\textbf{n})$$ in $\mathcal{G}_k$, where $d_{i_1}$, $\ldots$, $d_{i_{k-2}}$ is a contiguous sequence of distinguished diagonals  in  $\D(\textbf{u}_k)$. Since we flip all the distinguished diagonals of $\D(\textbf{u}_k)$ {\em contiguously} to obtain $\D(\textbf{n})$, the last distinguished diagonal that is flipped must be either $d_1$ or $d_{k-2}$.  In Example~\ref{ex: cont}, for instance, the last distinguished diagonal that is flipped is $d_1$ (see Figure~\ref{fig: ex}).

In the following, for ease of notation, we set $\textbf{n}' = \textbf{n}_{k-3}$. Note that $\D(\textbf{n}')$ is  obtained from $\D(\textbf{u}_k)$ by flipping the diagonals $d_{i_1}, \ldots, d_{i_{k-3}}$, in order. Our proof naturally splits into two possible cases.

{\bf Case A:} Suppose that $d_{i_{k-2}}=d_{1}$. In this case, we observe that  $d_{i_{k-2}}=d_{1}$ is geometrically in the {\em leftmost} position. Now, we peel away from  $\D(\textbf{n}')$ the "upper left" triangle whose only interior edge is $d_{1}$ and denote the resulting triangulation of $\mathcal{P}_{k}$ as $\D(\overline{\textbf{n}}')$. Note that $\overline{\textbf{n}}'$ has only one component that is equal to $1$, which is in the interior of $\overline{\textbf{n}}'$, by the first assumption in the lemma. It follows that the $k$-tuple $\textbf{n}$ is obtained from the $(k-1)$-tuple $\overline{\textbf{n}}'$ by increasing the last component of $\overline{\textbf{n}}'$ by $1$ and inserting $2$ at the beginning.

Let $\textbf{m}$ denote the $k$-tuple having $1$ in the position that $\textbf{n}$ has
a $1$ and $0$ elsewhere, and similarly let $\overline{\textbf{m}}'$ denote the $(k-1)$-tuple having $1$ in the position that $\overline{\textbf{n}}'$ has
a $1$ and $0$ elsewhere. Let $1\leq q < p $ be the unique coprime integers such that $\hj(\textstyle{\frac{p}{p-q}})=\textbf{n}+ \textbf{m}$, and similarly let $1 \leq \overline{q}' < \overline{p}'$ be the unique coprime integers such that $\hj(\textstyle{\frac{\overline{p}'}{\overline{p}'-\overline{q}'}})= \overline{\textbf{n}}'+ \overline{\textbf{m}}'$. It follows by the Riemenschneider's point diagram method that $\hj(\textstyle{\frac{p}{q}})$ is obtained from $\hj(\textstyle{\frac{\overline{p}'}{\overline{q}'}})$ by inserting $2$ at the end and increasing the first component by $1$.

{\bf Case B:} Suppose that $d_{i_{k-2}}=d_{k-2}$. In this case, we observe that  $d_{i_{k-2}}=d_{k-2}$ is geometrically in the {\em rightmost} position. Now, we peel away from  $\D(\textbf{n}')$ the "upper right" triangle whose only interior edge is $d_{k-2}$ and denote the resulting triangulation of $\mathcal{P}_{k}$ as $\D(\overline{\textbf{n}}')$. Note that $\overline{\textbf{n}}'$ has only one component that is equal to $1$, which is in the interior of $\overline{\textbf{n}}'$, by the first assumption in the lemma. It follows  that the $k$-tuple $\textbf{n}$ is obtained from the $(k-1)$-tuple $\overline{\textbf{n}}'$ by increasing the first component of $\overline{\textbf{n}}'$ by $1$ and inserting $2$ at the end.

Let $\textbf{m}$ denote the $k$-tuple having $1$ in the position that $\textbf{n}$ has
a $1$ and $0$ elsewhere, and similarly let $\overline{\textbf{m}}'$ denote the $(k-1)$-tuple having $1$ in the position that $\overline{\textbf{n}}'$ has
a $1$ and $0$ elsewhere. Let $1\leq q < p $ be the unique coprime integers such that $\hj(\textstyle{\frac{p}{p-q}})=\textbf{n}+ \textbf{m}$, and similarly let $1 \leq \overline{q}' < \overline{p}'$ be the unique coprime integers such that $\hj(\textstyle{\frac{\overline{p}'}{\overline{p}'-\overline{q}'}})= \overline{\textbf{n}}'+ \overline{\textbf{m}}'$. It follows by the Riemenschneider's point diagram method that $\hj(\textstyle{\frac{p}{q}})$ is obtained from the $\hj(\textstyle{\frac{\overline{p}'}{\overline{q}'}})$ by inserting $2$ at the beginning and increasing the last component by $1$.

The proof will be completed by an easy inductive argument. For the initial step of the induction, consider the case $k=3$. In this case, by flipping the unique distinguished diagonal of the quadrilateral, we obtain $\D(2,1,2)$ from the initial triangulation $\D(1,2,1)$. In this case, $\textbf{n}=(2,1,2)$, $\textbf{n}+ \textbf{m}=(2,2,2)$ and hence $\hj(\textstyle{\frac{p}{q}})=(4)$.  It should be clear that {\bf Case A} provides the inductive step corresponding to iteration $(a)$, whereas {\bf Case B} provides the inductive step corresponding to iteration $(b)$.
\end{proof}

\begin{Ex} {\em In Example~\ref{ex: cont}, we obtained $\D(2,2,2,4,2,1,3,2,5)$  from the initial triangulation $\D(\textbf{u}_9)$ by applying flips along the contiguous sequence  $d_5$, $d_4$, $d_6$, $d_7$, $d_3$, $d_2$, $d_1$ of distinguished diagonals of the initial triangulation of the decagon. If we run the algorithm in  Lemma~\ref{lem: plumbing}, based on the sequence $d_5$, $d_4$, $d_6$, $d_7$, $d_3$, $d_2$, $d_1$  we get $(4)$, $(5,2)$, $(2,5,3)$, $(2,2,5,4)$,  $(3,2,5,4,2)$,  $(4,2,5,4,2,2)$,  $(5,2,5,4,2,2,2)$. Therefore we conclude that, if $\hj(\textstyle{\frac{p}{p-q}}) = (2,2,2,4,2,2,3,2,5)$, then $\hj(\textstyle{\frac{p}{q}})$ must be equal to $(5,2,5,4,2,2,2)$, which can indeed be verified by the Riemenschneider's point diagram method as illustrated in Figure~\ref{fig: point}.

\begin{figure}[ht]\small {\epsfxsize=1.5in
\centerline{\epsfbox{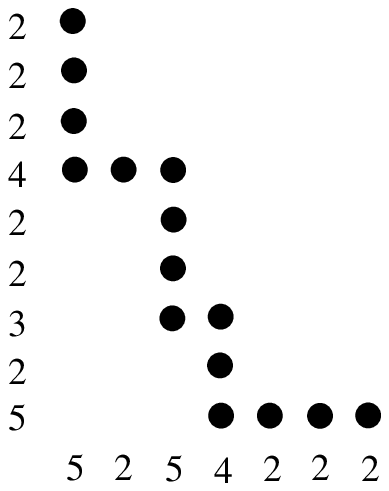}}}
\caption{Riemenschneider's point diagram method.} \label{fig: point} \end{figure}
}
\end{Ex}

\begin{Rem} \label{rem: arb} {\em  In Lemma~\ref{lem: plumbing}, we assumed that  $\textbf{n}$ is of maximal height, and minimum positive depth, i.e.,  $\hgt  (\textbf{n} ) = k-2$ and $\dpt  (\textbf{n} ) = 1$. In fact, we could formulate a similar result if $\textbf{n}$ is of arbitrary positive height, say $1 \leq s \leq k-2$ and minimum positive depth, as follows. The assumptions $\dpt  (\textbf{n} ) = 1$ and $\hgt  (\textbf{n} ) = s$ implies that there is a unique contiguous path of length $s$ from the root vertex $\D(\textbf{u}_k)$ to $\D(\textbf{n})$ in the graph $\mathcal{G}_k$, by Lemma~\ref{lem: cont}. Since the path is contiguous, we can peel away the irrelevant triangles form each of the triangulations in this path, to get a new path of the same length  in  $\mathcal{G}_{s+2}$, which starts from $\D(\textbf{u}_{s+2})$ and ends with a triangulation of maximal possible height and minimum positive depth. Then we apply Lemma~\ref{lem: plumbing} to this contiguous path in $\mathcal{G}_{s+2}$.
}
\end{Rem}

\begin{Ex} {\em Here we give an example to  illustrate Remark~\ref{rem: arb}. Consider the triangulation $\D((2,2,1,4,1))$ of the hexagon in the graph $\mathcal{G}_5$ depicted in Figure~\ref{fig: hexagon1}. The root vertex $\D(\textbf{u}_5)$ is connected to the vertex $\D((2,2,1,4,1))$ by the unique contiguous path
$$\D(\textbf{u}_5) \xrightarrow[\text{flip}]{\text{$d_{2}$}} \D((1,3,1,3,1)) \xrightarrow[\text{flip}]{\text{$d_{1}$}} \D((2,2,1,4,1))$$
of length $2$, obtained by concatenating the edges $e_1$ and $e_{2,1}$. Now, by removing the "top right" triangle from each of the triangulations in this path, we obtain a new path $$\D(\textbf{u}_4) \xrightarrow[\text{flip}]{\text{$\overline{d}_{2}$}} \D((1,3,1,2)) \xrightarrow[\text{flip}]{\text{$\overline{d}_{1}$}} \D((2,2,1,3))$$ of maximal possible length in the graph $\mathcal{G}_4$. Note that $$\dpt((2,2,1,4,1))= \dpt(((2,2,1,3))=1,$$ and hence  Lemma~\ref{lem: plumbing} can be applied to this new path in  $\mathcal{G}_4$. So, if $\hj(\textstyle{\frac{p}{p-q}})=(2,2,2,3)$, then $\hj(\textstyle{\frac{p}{q}}) = (5,2)$ which indeed belongs to $\mathcal{O}^{2}$ and $(5,2)$ is obtained from $(4)$ by the iteration of type $(b)$.
}
\end{Ex}

\section{Contiguous sequences of diagonal flips and rational blowdowns}

In this section, our goal is to prove Theorem~\ref{thm: main}, which essentially organizes the symplectic deformation  equivalence classes of all minimal symplectic  fillings of the contact lens space $(L(p,q), \xi_{can})$ as a graded, directed, rooted, connected graph, where the root is the minimal resolution of the corresponding cyclic quotient singularity and  each directed edge is a symplectic rational blowdown along a linear plumbing graph.

A rational blowdown is the surgery operation which replaces the neighborhood of a configuration of spheres in a smooth $4$-manifold intersecting according to some connected plumbing graph, by a rational homology ball having the same oriented boundary. Each vertex in a plumbing graph represents a disk bundle over the sphere and is decorated by the Euler number of the bundle, which is called the {\em weight} of the vertex.

\begin{Prop}[Wahl \cite{wa}, Looijenga-Wahl \cite{lw}] \label{prop: wahl} A linear plumbing graph can be rationally blown down if and only if the weights of its vertices are exactly given by taking the negatives of the entries in the Hirzebruch-Jung continued fraction expansion of $s^2 / (sh-1)$ for some pair of coprime integers $(s,h)$ with  $1 \leq h < s$. More explicitly, the family of linear plumbing graphs that can be rationally blown down is obtained from the initial graph with one vertex whose weight is $-4$, and applying the following iterations: If  the linear plumbing graph with weights $-a_1,  \ldots, -a_r$ is in this family so are  the linear plumbing graphs with weights
\begin{enumerate}[label=(\Roman*)]
\item {$-2, -a_1,  \ldots, -a_{r-1}, -(a_r +1)$ and }
\item {$-(a_1+1),  -a_2, \ldots,  -a_r,  -2$. }
\end{enumerate}
\end{Prop}

In the context of $4$-manifolds, the rational blowdowns along linear plumbing graphs, were first used by Fintushel and Stern \cite{fs} for the case $h=1$, and by Park \cite{p} for the general case.
From the singularity theory point of view, each of these linear  plumbing graphs is the dual minimal resolution graph of some cyclic quotient singularity of class $T_0$ (a.k.a. {\em Wahl singularity}), which is a subclass of singularity of class $T$ (see \cite{ksb}).

Moreover, Symington (\cite{sy1, sy2})  established that  the rational blowdown surgery preserves
a symplectic structure if the original spheres are symplectic surfaces in a symplectic $4$-manifold.

Next we recall some definitions which will be used in the proof of Theorem~\ref{thm: main} below. For any $\textbf{n}=(n_1, \ldots, n_k)
\in \mathcal{Z}_{k}$, let $N(\textbf{n})$ denote the result of Dehn
surgery on the framed link which consists of the chain of $k$ unknots in $S^3$ with framings $n_1, n_2,
\ldots, n_k$, respectively. It follows easily that the $3$-manifold $N(\textbf{n})$
is diffeomorphic to $S^1 \times S^2$. Let $\textbf{m}= (m_1,\ldots,m_k)\in
\mathbb{Z}^k$, and
$\textbf{L}=\bigcup_{i=1}^{k} L_i$ denote the framed link in
$N(\textbf{n})$, in the complement of the
chain of $k$ unknots, where each $L_i$ consists of $|m_i|$
components as depicted in Figure~\ref{fig: han}, with the components having framings $-1$
if $m_i>0$ and framings $+1$ if $m_i<0$.

\begin{figure}[ht]
  \relabelbox \small {
  \centerline{\epsfbox{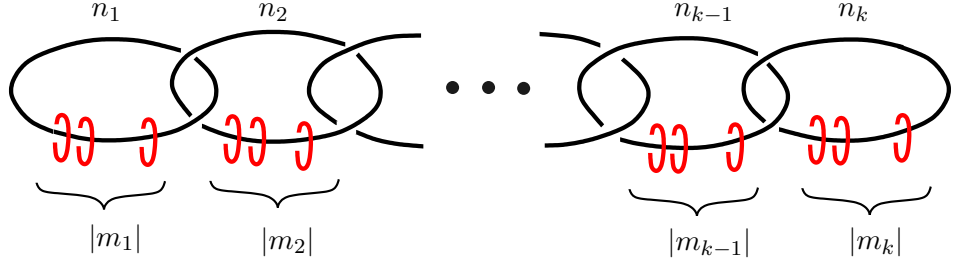}}}
\relabel{a}{$n_1$} \relabel{b}{$n_2$} \relabel{d}{$n_{k-1}$}
\relabel{e}{$n_k$} \relabel{1}{$|m_1|$} \relabel{2}{$|m_2|$} \relabel{3}{$|m_{k-1}|$} \relabel{4}{$|m_k|$}
\endrelabelbox
       \caption{The relative handlebody description of the oriented smooth $4$-manifold $W(\textbf{n},\textbf m)$.}
 \label{fig: han}
\end{figure}

\begin{Def} \label{wnm} For any
$\textbf{n}=(n_1,\ldots,n_k)\in\mathcal{Z}_{k}$, and  $\textbf{m}= (m_1,\ldots,m_k)\in
\mathbb{Z}^k$, the oriented smooth $4$-manifold $W(\textbf{n},\textbf m)$ is obtained by attaching
$2$-handles to $S^1 \times D^3$ along the framed link $\varphi(\textbf{L}) \subset S^1
\times S^2$ for some diffeomorphism $\varphi : N(\textbf{n})\to S^1 \times S^2$.
\end{Def}

Note that this description, which is independent of the choice of $\varphi$ since any self-diffeomorphism
of $S^1 \times S^2$ extends to $S^1 \times D^3$, is a \emph{relative }handlebody decomposition of
 $W(\textbf{n},\textbf m)$.

\begin{Def} \label{def: achiral}  Let $\textbf{b}=(b_1, \ldots, b_k)= \hj(\textstyle{\frac{p}{p-q}})$, where $\frac{p}{p-q}=[b_1, b_2, \ldots, b_k]$. For any $\textbf{n} \in \mathcal{Z}_k$, we set $W_{p,q} (\textbf{n}) =   W (\textbf{n}, \textbf{b-n}).$  \end{Def}

According to Lisca's classification \cite{l}, any minimal symplectic filling of the contact $3$-manifold $(L(p,q), \xi_{can})$ is orientation-preserving diffeomorphic to  $W_{p,q} (\textbf{n}) $ for some $\textbf{n} \in \mathcal{Z}_{k}(\textstyle{\frac{p}{p-q}})$, and the symplectic structure on $W_{p,q} (\textbf{n}) $ is unique up to symplectic deformation equivalence \cite{bono}.

\begin{Rem} \label{rem: rhb} {\em Fix an integer $k \geq 3$ and let  $\textbf{n} \in \mathcal{Z}_k$ such that exactly
one component of $\textbf{n}$  equals to $1$, which is in the interior of $\textbf{n}$.
  Let $\textbf{m}$ denote the $k$-tuple having $1$ in the position that $\textbf{n}$ has a $1$ and $0$ elsewhere, and let $1\leq q < p $ be the unique coprime integers such that $\hj(\textstyle{\frac{p}{p-q}})=\textbf{n}+ \textbf{m}$. We proved in \cite{bo}  that  the minimal symplectic filling $W_{p,q}(\textbf{n})$ of $(L(p,q),\xi_{can})$  is a rational homology $4$-ball and thus $W_{p,q}(\textbf{n})$ can be obtained from the canonical symplectic filling $W_{p,q}(\textbf{u}_k)$ by a single symplectic rational blowdown along a linear plumbing graph. Moreover, the weights of the linear plumbing graph are given by the negatives of the components of $\hj(\textstyle{\frac{p}{q}})$.} \end{Rem}

We are now ready to give a proof of  Theorem~\ref{thm: main}.

\smallskip

\begin{proof}[Proof of Theorem~\ref{thm: main}.]  Suppose that $p > q \geq 1$ are coprime integers and let
$\frac{p}{p-q}=[b_1, b_2, \ldots, b_k]$ be the Hirzebruch-Jung continued fraction expansion, where $b_i \geq 2$ for $1 \leq i \leq
k$.

We will construct a graded, directed, rooted, connected graph $\mathcal{G}^{p,q}_k$ satisfying items (1) to (4) in Theorem~\ref{thm: main} using the graded, directed, rooted,  connected graph $\mathcal{G}_k$. Note that the root of  $\mathcal{G}_k$ is the initial triangulation $\D_\star = \D (\textbf{u}_k)$, and $W_{p,q} (\textbf{u}_k)$ is the minimal resolution, which is the canonical symplectic filling of $(L(p,q), \xi_{can})$.  We take $\D_\star$ as the root of  $\mathcal{G}^{p,q}_k$. Each vertex of $\mathcal{G}_k$ can be identified with an element of $T(\mathcal{P}_{k+1})$ and Stevens' bijection $\Phi_k$ identifies $T(\mathcal{P}_{k+1})$ with the set $\mathcal{Z}_k$. Moreover, the  symplectic deformation  classes of all minimal symplectic fillings of the contact $3$-manifold $(L(p,q), \xi_{can})$ is parametrized by the subset $\mathcal{Z}_{k}(\textstyle{\frac{p}{p-q}})$ of $\mathcal{Z}_k$ or, equivalently, by the subset $ \mathcal{T}^{p,q} (\mathcal{P}_{k+1})$ of $\mathcal{T} (\mathcal{P}_{k+1})$.  So, to obtain the vertices of $\mathcal{G}^{p,q}_k$, we just take the vertices of $\mathcal{G}_k$ which belong to $ \mathcal{T}^{p,q} (\mathcal{P}_{k+1})$ and ``skip" the others. So far, our graph $\mathcal{G}^{p,q}_k$ satisfies items (1) and (2) in Theorem~\ref{thm: main}.

We now turn our attention to item (3). Each edge in the graph $\mathcal{G}^{p,q}_k$ is obtained by the concatenation of some edges in $\mathcal{G}_k$ according to the following principle:

{\em Suppose that  $\D (\textbf{n})$ and $\D (\textbf{n}')$ are two vertices in the graph $\mathcal{G}^{p,q}_k$.  A path of  directed edges in $\mathcal{G}_k$ from  $\D (\textbf{n})$ to $\D (\textbf{n}')$ are concatenated into a single directed edge in $\mathcal{G}^{p,q}_k$  if and only if  there is a unique path from $\D (\textbf{n})$ to $\D (\textbf{n}')$ in $\mathcal{G}_k$. }

We check that with this convention, if there is a directed edge from $\D (\textbf{n})$ to $\D (\textbf{n}')$ in $\mathcal{G}^{p,q}_k$, then the
minimal symplectic filling $W_{p,q} (\textbf{n}')$ can be obtained from the minimal symplectic filling $W_{p,q} (\textbf{n})$ by a single rational
blowdown along a linear plumbing graph. So suppose that $\D (\textbf{n})$ is connected to $\D (\textbf{n}')$ by a unique
path in $\mathcal{G}_k$. Let $\textbf{n}=\textbf{n}_0,\textbf{n}_1,\ldots,\textbf{n}_r=\textbf{n}'$ be the
sequence of $k$-tuples corresponding to the vertices in this path, and let $d_{i_1},d_{i_2},\ldots,d_{i_r}$ be the corresponding sequence of distinguished diagonals of $\mathcal{P}_{k+1}$ that are flipped along the edges, as illustrated below:

 $$\D(\textbf{n})=\D(\textbf{n}_0) \xrightarrow[\text{flip}]{\text{$d_{i_1}$}} \D(\textbf{n}_1) \xrightarrow[\text{flip}]{\text{$d_{i_2}$}} \D(\textbf{n}_2) \rightarrow \cdots \rightarrow \D(\textbf{n}_{r-1})  \xrightarrow[\text{flip}]{\text{$d_{i_r}$}} \D(\textbf{n}_r)=\D(\textbf{n}').$$

We claim that $d_{i_1}, d_{i_2}, \ldots, d_{i_r}$ is a {\em contiguous} sequence of distinguished diagonals in the triangulation $\D(\textbf{n})$. To see this,  note that each pair of successive distinguished diagonals $d_{i_j},d_{i_{j+1}}$ that are flipped must be adjacent in the sense each such pair of diagonals must bound a triangle in $\D (\textbf{n}_{j-1})$. Indeed if $d_{i_j}$ and $d_{i_{j+1}}$ were not adjacent then
we could find an alternate path from $\D (\textbf{n})$ to $\D (\textbf{n}')$ in $\mathcal{G}_k$ by interchanging the order in which we
flip $d_{i_j}$ and $d_{i_{j+1}}$, contradicting the uniqueness of the path between $\D (\textbf{n})$ to $\D (\textbf{n}')$ in $\mathcal{G}_k$. Let
$\mathcal{K}_r =\{d_{i_1}, d_{i_2}, \ldots, d_{i_r}\}$ and peel away all triangles from $\D(\textbf{n})$ that do not have an edge in the set $\mathcal{K}_r$. This
will transform the polygon $\mathcal{P}_{k+1}$ into an $(r+3)$-gon  $\mathcal{P}_{r+3}$.

Moreover,  the triangulation $\D (\textbf{n})$ of $\mathcal{P}_{k+1}$
will become the initial triangulation $\D (\textbf{u}_{r+2})$ of $\mathcal{P}_{r+3}$ and the contiguous sequence $d_{i_1}, d_{i_2}, \ldots, d_{i_r}$ of distinguished diagonals in $\D(\textbf{n})$ will become
a contiguous sequence  of distinguished diagonals of $\D (\textbf{u}_{r+2})$. Let $\textbf{u}_{r+2}=\overline{\textbf{n}}_0,\overline{\textbf{n}}_1,\ldots,\overline{\textbf{n}}_r$ be the sequence of $(r+2)$-tuples corresponding to the sequence of
triangulations of $\mathcal{P}_{r+3}$ obtained from the triangulations $\D(\textbf{n}_0),\D(\textbf{n}_1),\ldots,\D(\textbf{n}_r)$ of $\mathcal{P}_{k+1}$ by peeling away the same set of triangles as above. By Lemma~\ref{lem: plumbing}, it follows that $(r+2)$-tuple $\overline{\textbf{n}}_r$ will have exactly one component that is equal to $1$, which is in the interior of $\overline{\textbf{n}}_r$. We illustrated this step of the proof in Example~\ref{ex: path} below.

Let $\overline{\textbf{m}}_r$ be the $(r+2)$-tuple having $1$ in the position that $\overline{\textbf{n}}_r$ has
a $1$ and $0$ elsewhere, and let $\overline{\textbf{b}}_r=\overline{\textbf{m}}_r + \overline{\textbf{n}}_r$. Let $p'$ and $q'$ be the coprime integers with $1 \leq q' < p'$ such that $\hj(\textstyle{\frac{p'}{p'-q'}})=\overline{\textbf{b}}_r$. According to Remark~\ref{rem: rhb}, the minimal symplectic filling $W(\overline{\textbf{n}}_r,\overline{\textbf{m}}_r) = W_{p',q'}(\overline{\textbf{n}}_r)$ of $(L(p',q'), \xi_{can})$ is a rational homology ball and moreover, it is obtained from  the canonical symplectic filling $W(\textbf{u}_{r+2},\overline{\textbf{b}}_r -\textbf{u}_{r+2}) = W_{p',q'}(\textbf{u}_{r+2})$ of $(L(p', q'), \xi_{can})$  by a single rational blowdown  along a linear plumbing graph. Moreover, the weights of the linear plumbing graph are given by the negatives of  the components in $\hj(\textstyle{\frac{p'}{q'}})$.

Since the triangulation $\D(\textbf{n}_i)$ of $\mathcal{P}_{k+1}$ is obtained from the triangulation $\D(\overline{\textbf{n}}_i)$ of $\mathcal{P}_{r+3}$ by pasting on the collection of triangles we peeled away in the first place, it follows that the fibre of the planar Lefschetz fibration on $W(\overline{\textbf{n}}_i,\overline{\textbf{m}}_i)$, where $\overline{\textbf{m}}_i=\overline{\textbf{b}}_r-
\overline{\textbf{n}}_i$,  is canonically
embedded in the fibre of the planar Lefschetz fibration on $W_{p,q} (\textbf{n}_i) = W(\textbf{n}_i,\textbf{b}-\textbf{n}_i)$. Moreover, the monodromy of
 the planar Lefschetz fibration on $W(\overline{\textbf{n}}_i,\overline{\textbf{m}}_i)$ is contained as a subword in the monodromy of the planar Lefschetz fibration on $W_{p,q} (\textbf{n}_i)$. It follows
 that the minimal symplectic filling $W_{p,q} (\textbf{n}')= W_{p,q} (\textbf{n}_r)= W(\textbf{n}_r,\textbf{b}-\textbf{n}_r)$ can be obtained from the
minimal symplectic filling $W_{p,q} (\textbf{n})= W_{p,q} (\textbf{n}_0) = W(\textbf{n}_0,\textbf{b}-\textbf{n}_0)$ by a single rational blowdown along a linear plumbing graph as claimed.

Finally, to prove the claim in item (4), rather than computing the second Betti number of the minimal symplectic filling $W_{p,q} (\textbf{n})$ directly, we compute instead  the Milnor number of the Milnor fibre,  which corresponds to the same parameter $\textbf{n}$. Denoting the Milnor fibre as $W_{p,q} (\textbf{n})$, by a slight abuse of notation, we recall  the simple formula for the Milnor number $$\mu (W_{p,q} (\textbf{n})) = r+2(k-1)-|\textbf{n}|$$ of  $W_{p,q} (\textbf{n})$, where $r$  is the length of the   Hirzebruch-Jung continued fraction expansion of $\textstyle{\frac{p}{q}}$ and $|\textbf{n}|= n_1+ \cdots + n_k.$  For the formula of the Milnor number, we refer the reader to \cite[Theorem 7.7]{n} and references therein. Since we fix the pair $(p,q)$ from the beginning of the proof, $r$ and $k$ are fixed and  hence $\mu (W_{p,q} (\textbf{n}))$ only depends on $|\textbf{n}|$. But it is easy to see (as in the proof of Lemma~\ref{lem: hgt}) that  $|\textbf{n}|$ increases by one after applying a diagonal flip along a distinguished diagonal, and thus $\mu (W_{p,q} (\textbf{n}))$ drops by one.  Therefore, if the grading of each vertex of the graph $\mathcal{G}^{p,q}_k$ is defined as the second Betti number of the corresponding minimal symplectic filling, then each directed edge drops the grading by the number of diagonal flips used to obtain that edge as described above.  We also note that $\mu(W_{p,q} (\textbf{n})) = r - \hgt(\textbf{n})$, where $\hgt$ is the height function described in Definition~\ref{def: height}. See Remark~\ref{rem: alt}, for an alternative direct proof of item (4).  \end{proof}

\begin{Ex} \label{ex: path} {\em We illustrate a crucial step in the proof of Theorem~\ref{thm: main} by the following example. Consider the path
$$\D(\textbf{n}) = \D((1,2,2,4,2,1,3,3,1)) \xrightarrow[\text{flip}]{\text{$d_7$}}  \D(\textbf{n}_1) = \D((1,2,2,5,2,1,3,2,2)) \xrightarrow[\text{flip}]{\text{$d_{3}$}} $$ $$\D(\textbf{n}_2) = \D((1,2,3,4,2,1,3,2,3)) \xrightarrow[\text{flip}]{\text{$d_{2}$}} \D(\textbf{n}_3) = \D((1,3,2,4,2,1,3,2,4))$$ of triangulations of the decagon  as depicted at the top row in Figure~\ref{fig: path}. In this example, the set
$\mathcal{K}_3$ mentioned in the proof of Theorem~\ref{thm: main}, consists of the {\em contiguous} sequence $d_{7}, d_{3}, d_{2}$ of distinguished diagonals
in the triangulation $\D(\textbf{n})$. By peeling away  all triangles from $\D(\textbf{n})$ that do not have an edge in the set $\mathcal{K}_3$, we obtain the initial triangulation $\D(\textbf{u}_5)$ of the hexagon, as depicted at the beginning of the bottom row in Figure~\ref{fig: path}.

\begin{figure}[ht]\small {\epsfxsize=6in
\centerline{\epsfbox{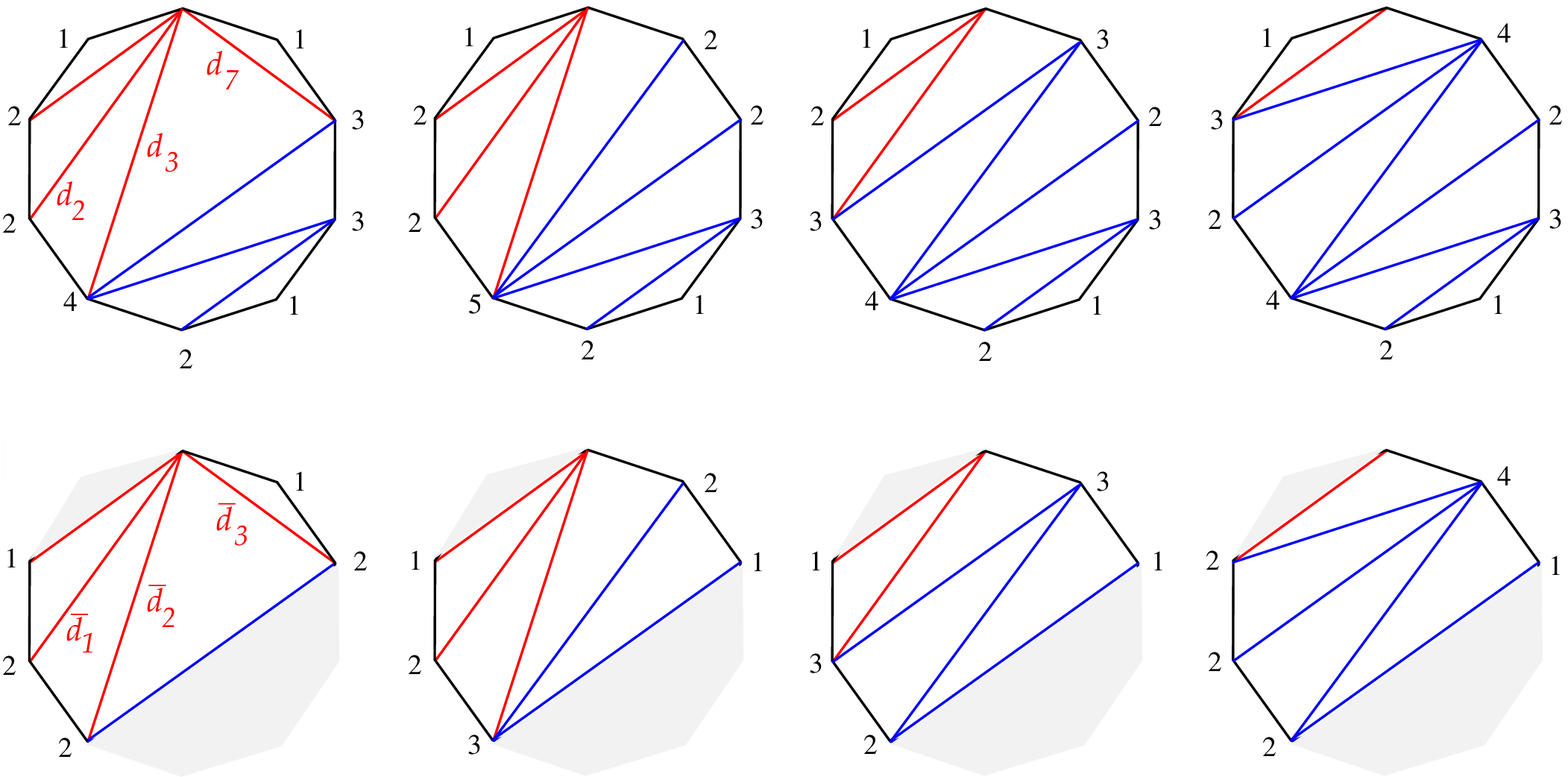}}}
\caption{Top row: $\D(\textbf{n})$, $\D(\textbf{n}_1)$, $\D(\textbf{n}_2)$, $\D(\textbf{n}_3)$. Bottom row: $\D(\overline{\textbf{n}})$, $\D(\overline{\textbf{n}}_1)$, $\D(\overline{\textbf{n}}_2)$, $\D(\overline{\textbf{n}}_3)$. } \label{fig: path} \end{figure}

Moreover, by peeling away the same triangles form each of the triangulations $\D(\textbf{n})$, $\D(\textbf{n}_1)$, $\D(\textbf{n}_2)$, $\D(\textbf{n}_3)$ of the decagon depicted at the top row, we obtain the path of triangulations $$\D(\overline{\textbf{n}})=\D(\textbf{u}_{5})  \xrightarrow[\text{flip}]{\text{$\overline{d}_3$}} \D(\overline{\textbf{n}}_1)  = \D((1,2,3,1,2)) \xrightarrow[\text{flip}]{\text{$\overline{d}_2$}} $$ $$  \D(\overline{\textbf{n}}_2)=\D((1,3,2,1,3)) \xrightarrow[\text{flip}]{\text{$\overline{d}_1$}} \overline{\textbf{n}}_3 =\D((2,2,2,1,4))$$ of the hexagon depicted in the bottom row in Figure~\ref{fig: path}. Here, we used overline for the distinguished diagonals of the hexagon to set them apart from the distinguished diagonals of the decagon. It should be clear that the contiguous sequence $d_{7}, d_{3}, d_{2}$ of distinguished diagonals  of $\D(\textbf{n})$ can be identified with the  the contiguous sequence $\overline{d}_3,\overline{d}_2, \overline{d}_1$ of distinguished diagonals  of $\D(\overline{\textbf{n}})=\D(\textbf{u}_{5})$. Note that $\overline{\textbf{n}}_3$ has exactly one component that is equal to $1$, which is in the interior of $\overline{\textbf{n}}_3$. We emphasize that the interior  $1$ in $\overline{\textbf{n}}_3$ is obtained by the diagonal flip along the distinguished diagonal $\overline{d}_3$ of the hexagon, which is applied first in the sequence of diagonal flips in the bottom row of triangulations in Figure~\ref{fig: path}.}
\end{Ex}

\begin{Prop}\label{prop: lpg} Suppose that there is a directed edge from some vertex $\D(\textbf{n})$ to another vertex $\D(\textbf{n}')$ in  $\mathcal{G}^{p, q}_k$ which is obtained by the unique contiguous path in  $\mathcal{G}_k$ corresponding to the sequence of diagonal flips along the distinguished diagonals $d_{i_1}, \ldots, d_{i_r}$. Then the linear plumbing graph for the rational blowdown represented by this edge is obtained by starting from the initial graph with one vertex whose weight is $-4$ and applying  the iterations in Proposition~\ref{prop: wahl} according to the following rule:  If $i_j < i_{j+1}$, then apply iteration $(I)$, otherwise apply iteration $(II)$. \end{Prop}

\begin{proof} Using the same argument (and notation) as in the proof of Theorem~\ref{thm: main}, we see that the linear plumbing graph used for the rational blowdown that yields the minimal symplectic filling $W_{p,q} (\textbf{n})$ from the minimal symplectic  filling $W_{p,q} (\textbf{n}')$ is the same as the linear plumbing graph used for the rational blowdown that yields the minimal symplectic filling $W_{p',q'} (\overline{\textbf{n}}_r)$ from the  canonical symplectic filling $W_{p',q'} (\textbf{u}_{r+2})$. Note that the triangulation  $\D(\overline{\textbf{n}}_r)$  of $\mathcal{P}_{r+3}$  is obtained from the initial triangulation $\D(\textbf{u}_{r+2})$ by flipping all the distinguished diagonals in $\D(\textbf{u}_{r+2})$ {\em contiguously.} As a consequence, the proof of Proposition~\ref{prop: lpg} reduces to Lemma~\ref{lem: plumbing}.
\end{proof}

\section{Diagonal flips and lantern substitutions} \label{sec: lantern}

Suppose $p$ and $q$ are coprime integers with $p >q \geq 1$ such that the Hirzebruch-Jung continued fraction expansion of $\textstyle{\frac{p}{p-q}}$
 is equal to $[b_1, \ldots, b_k]$, where $b_i \geq 2$ for all $1 \leq i \leq k$.   In \cite{bo}, we constructed a planar Lefschetz fibration on each minimal symplectic filling of the contact $3$-manifold  $(L(p,q), \xi_{can})$. In particular, there is a Lefschetz fibration on the minimal resolution, whose fibre is the disk $D_k$ with $k$-holes and whose monodromy is the composition of Dehn twists along an explicit set of {\em disjoint} curves in $D_k$. We would like to point out that the planar Lefschetz fibration constructed by Gay and Mark \cite{gm} on the minimal resolution using its dual plumbing graph agrees with ours.

Moreover, the planar Lefschetz fibration above  naturally induces a {\em planar}  open book $\OB_{p,q}$ on $L(p,q)$ which supports $\xi_{can}$.   It follows by a general result of Wendl \cite{w}, that each  minimal symplectic filling of $(L(p,q), \xi_{can})$ has a planar Lefschetz fibration whose monodromy is a positive factorization of  the monodromy of $\OB_{p,q}$,  although we have not relied on his result in \cite{bo}.

Furthermore, in \cite[Theorem 4.1]{bo}, we showed that each minimal symplectic filling of the contact $3$-manifold  $(L(p,q), \xi_{can})$ can be obtained from the minimal resolution by a sequence of rational blowdowns along linear plumbing graphs. We observe here that the proof of Lemma 4.5 in \cite{bo},  coupled with Lemma~\ref{lem: hgt} of the present paper,  implies in particular that if $\D(\widetilde{\textbf{n}}) \in \mathcal{T} (\mathcal{P}_{k+1}) $  is obtained from $\D(\textbf{n}) \in \mathcal{T} (\mathcal{P}_{k+1})$ by a diagonal flip, then the monodromy of the possibly {\em achiral} planar Lefschetz fibration on $W_{p,q} (\textbf{n})$ can be obtained from the monodromy of the possibly {\em achiral} planar   Lefschetz fibration on $W_{p,q} (\widetilde{\textbf{n}})$    by single lantern substitution together with, possibly, the introduction or removal of some cancelling pairs of Dehn twists.

The reader might be puzzled at this point at why we allow achiral Lefschetz fibrations in this discussion, but the point is that we can go from a positive factorization of some fixed monodromy to another positive factorization by a sequence of lantern substitutions which destroys positivity at the intermediate steps but restores it at the end.

As a matter of fact, the lantern substitution is completely determined by the diagonal flip,  which we discuss below. The following definition is needed in our discussion.

  \begin{Def} \label{def: curves} Let $D_k$ denote the disk with $k$-holes. Suppose that the holes in $D_k$ are aligned horizontally and enumerated from left  to right. For each $1 \leq r \leq k$, let $\g_r$ denote the convex curve enclosing the first $r$ holes and for any $2 \leq s \leq t \leq k$, let $\d_{s,t}$ denote the convex curve enclosing the holes labelled from $s$ to $t$.    \end{Def}

 The diagonal flip along any distinguished diagonal in any given triangulation $\D(\textbf{n})$ of $\mathcal{P}_{k+1}$ transforms $\D(\textbf{n})$   to  another triangulation $\D(\widetilde{\textbf{n}})$ of $\mathcal{P}_{k+1}$,  so that for exactly two indices,  say $i < j$, we have $\widetilde{n}_i - n_i  = \widetilde{n}_j - n_j= 1$ and for one index, say $t$,  where $i < t <j$,  we have $\widetilde{n}_t - n_t =-1$.  This is simply because each diagonal flip is an exchange of a distinguished diagonal with a non-distinguished diagonal of a quadrilateral, one of whose vertices is the distinguished vertex.  The corresponding lantern substitution in the monodromy factorization of the planar Lefschetz fibration on $W_{p,q} (\textbf{n})$, in order to obtain the monodromy factorization of the planar Lefschetz fibration on $W_{p,q} (\widetilde{\textbf{n}})$,  is the replacement of the product of four Dehn twists $$D(\g_i) \circ D(\d_{i+1, t}) \circ D(\d_{t+1, j}) \circ D(\g_j)$$ with the product of three Dehn twists
$$D(\b_{i,t,j}) \circ D(\g_{t}) \circ D(\d_{i+1, j}),$$ up to cyclic permutations, where $\b_{i,t,j}$ is depicted in Figure~\ref{fig: beta}.

\begin{figure}[ht]  \relabelbox \small {\epsfxsize=4in
\centerline{\epsfbox{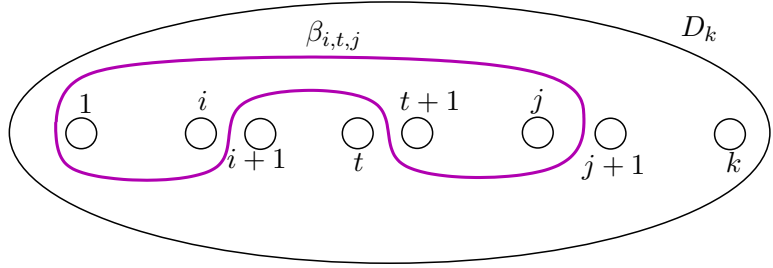}}}
 \relabel{1}{$1$} \relabel{2}{$i$} \relabel{3}{$i+1$} \relabel{4}{$t$} \relabel{5}{$t+1$} \relabel{6}{$j$} \relabel{7}{$j+1$} \relabel{8}{$k$}
 \relabel{9}{$\b_{i,t,j}$} \relabel{a}{$D_k$}
 \endrelabelbox
\caption{The curve $\b_{i,t,j}$ in $D_k$.} \label{fig: beta} \end{figure}

Note that if both $\textbf{n}$ and  $\widetilde{\textbf{n}}$ belong to  $\mathcal{Z}_{k}(\textstyle{\frac{p}{p-q}})$, then $W_{p,q} (\textbf{n})$ and  $W_{p,q} (\widetilde{\textbf{n}})$ both represent minimal symplectic fillings and  there is no need to insert any cancelling pair of Dehn twists to apply the lantern substitution. If $\textbf{n} \in \mathcal{Z}_{k}(\textstyle{\frac{p}{p-q}})$, but $\widetilde{\textbf{n}} \notin \mathcal{Z}_{k}(\textstyle{\frac{p}{p-q}})$, then to apply the lantern substitution, one needs to insert a cancelling pair of Dehn twists along  $\g_i$ if $\widetilde{n}_i > b_i$, and  a cancelling pair of Dehn twists along $\g_j$ if $\widetilde{n}_j > b_j$. In certain cases, both conditions are satisfied and we need to insert two cancelling pairs of Dehn twists. It is also possible that neither  $\textbf{n}$ nor  $\widetilde{\textbf{n}}$ belongs to  $\mathcal{Z}_{k}(\textstyle{\frac{p}{p-q}})$, in which case, one again inserts a cancelling pair of Dehn twists along  $\g_i$ or $\g_j$, or both, with the same criterion as above.

The upshot of this discussion is that, once the coprime pair $(p,q)$ is fixed,  each vertex in the graph $\mathcal{G}_k$ is a certain (not necessarily positive) factorization of the fixed monodromy of the planar open book $\OB_{p,q}$ on $L(p,q)$ which supports $\xi_{can}$. It follows that, each  vertex of $\mathcal{G}_k$ can be identified with a "smooth filling" $W_{p,q} (\textbf{n})$ of $L(p,q)$, for the corresponding $\textbf{n} \in  \mathcal{Z}_{k}$. Moreover, $W_{p,q} (\textbf{n})$ is a minimal symplectic filling of $(L(p,q), \xi_{can})$ if and only if the corresponding factorization is positive, or equivalently, if and only if  $\textbf{n}  \in \mathcal{Z}_{k}(\textstyle{\frac{p}{p-q}})$.

\begin{Rem}\label{rem: alt}  {\em An alternative proof of item $(4)$ in Theorem~\ref{thm: main}, which says that the Milnor number drops by one after each diagonal flip, in the spirit of the present paper, can be given as follows. By construction,  each diagonal flip is a lantern substitution in the monodromy of the corresponding planar Lefschetz fibration. By the work of Endo and Nagami \cite{en}, the signature $\sigma (W_{p,q} (\textbf{n})) = b_2^+ (W_{p,q} (\textbf{n})) -  b_2^- (W_{p,q} (\textbf{n})$ increases by one when a lantern substitution is applied, and hence $\mu(W_{p,q} (\textbf{n})) = b_2^+ (W_{p,q} (\textbf{n})) +  b_2^- (W_{p,q} (\textbf{n}))$ drops by one since $b_2^+ (W_{p,q} (\textbf{n}))$ remains fixed.} \end{Rem}

\section{Rational blowdown depth of a minimal symplectic filling}

Our goal in this section is to prove Proposition~\ref{prop: depth} from the Introduction. Recall that a minimal symplectic filling of $(L(p,q), \xi_{can})$ is said to have  {\em rational blowdown depth}  $r$ if the minimal number of successive symplectic rational blowdowns along linear plumbing graphs needed to obtain the filling from the minimal resolution is equal to $r$, where the depth of the minimal resolution is set to be zero. Moreover, for $k \geq 3$, the depth of  $\textbf{n} = (n_1, \ldots, n_k) \in \mathcal{Z}_{k}$, denoted by $\dpt (\textbf{n})$, is the number of $1$'s in the {\em interior} of $\textbf{n}$, i.e., $\dpt (\textbf{n})$ is the cardinality of the set $\{ i \; |\;  1 < i < k \; \mbox{and} \; n_i=1\}$.

\begin{Lem} \label{lem: same} Suppose $p$ and $q$ are coprime integers with $p >q \geq 1$. Let $\textbf{n}$ be a $k$-tuple in
$\mathcal{Z}_{k}(\textstyle{\frac{p}{p-q}})$ with $\dpt (\textbf{n})=l \geq 1$, so that $\textbf{n}$ has $l$ interior $1$'s enumerated from left to right. Pick any interior $1$, say the $v$th
interior $1$ for some $1\leq v\leq l$ and blow down $\textbf{n}=\textbf{n}^0$ at this $1$, resulting in a $(k-1)$-tuple by $\textbf{n}^1$. If  $\dpt(\textbf{n}^1)=l$, then blow down at the $v$th interior $1$ again to obtain a $(k-2)$-tuple $\textbf{n}^2$. Repeat in this way until the $(k-t)$-tuple $\textbf{n}^t$ satisfies  $\dpt(\textbf{n}^t)=l-1$. If $$ \D(\textbf{u}_k)=\D(\textbf{n}_0),\D(\textbf{n}_1),\ldots,\D(\textbf{n}_t)$$ is the corresponding path in $\mathcal{G}_k$ obtained as in Lemma~\ref{lem: once}, then $\textbf{n}_t$ belongs to $\mathcal{Z}_{k}(\textstyle{\frac{p}{p-q}})$.
  \end{Lem}

\begin{proof}
Suppose that $\textbf{n}$ is a $k$-tuple in $\mathcal{Z}_{k}(\textstyle{\frac{p}{p-q}})$ such that $\dpt (\textbf{n})=l$, so that $\textbf{n}$ has $l$ interior $1$'s enumerated from left to right. Pick any interior $1$, say the $v$th
interior $1$ for some $1\leq v\leq l$. Blow down $\textbf{n}=\textbf{n}^0$ at this $1$, and let $\tau_1$ denote the corresponding triangle
in $\D(\textbf{n})$ and $d_{i_1}$ the corresponding distinguished diagonal of  $\mathcal{P}_{k+1}$ as discussed in the proof of Lemma~\ref{lem: once}. Denote the resulting
$(k-1)$-tuple by $\textbf{n}^1$. Then $\textbf{n}^1$ will have at most $l$ interior $1$'s. If  $\textbf{n}^1$ again has $l$ interior $1$'s,
then blow down at the $v$th interior $1$ again to obtain a $(k-2)$-tuple $\textbf{n}^2$. Let $\tau_2$ denote the corresponding triangle
in $\D(\textbf{n})$ and $d_{i_2}$ the corresponding distinguished diagonal of $\mathcal{P}_{k+1}$. Repeat in this way until we obtain a
$(k-t)$-tuple $\textbf{n}^t$ with less than $l$ interior $1$'s. Let $\tau_1,\tau_2,\ldots,\tau_t$ and $d_{i_1},d_{i_2},\ldots,d_{i_t}$ denote
the associated sequences of triangles and diagonals, respectively. Then, as we blow down each time at sequentially the same interior
$1$, each pair of successive $\widetilde{d}_{i_j}$'s will be adjacent in the sense that they bound a triangle in $\D(\textbf{n})$. Let $\textbf{n}_t$
 be the $k$-tuple that corresponds to the triangulation obtained from the initial triangulation $\D(\textbf{u}_k)$ by flipping the distinguished diagonals $d_{i_1},d_{i_2},\ldots,d_{i_t}$ in order as in the proof of Lemma~\ref{lem: once}. Note that $d_{i_1},d_{i_2},\ldots,d_{i_t}$ is a contiguous sequence of distinguished diagonals in $\D(\textbf{u}_k)$.

Let $\textbf{b}$ be the $k$-tuple $(b_1, b_2, \ldots, b_k)$, where $\frac{p}{p-q}=[b_1, b_2, \ldots, b_k]$ is the Hirzebruch-Jung continued fraction expansion, with $b_i \geq 2$ for $1 \leq i \leq k$. Since, by assumption, $\textbf{n} \in \mathcal{Z}_{k}(\textstyle{\frac{p}{p-q}})$ we know that $\textbf{n} \leq \textbf{b}$. In the following we check that $\textbf{n}_t \leq \textbf{b}$, which implies that $\textbf{n}_t \in \mathcal{Z}_{k}(\textstyle{\frac{p}{p-q}})$ as well.

First note that the two triangulations $\D(\textbf{n}_t)$ and $\D(\textbf{n})$ both contain all of the triangles $\tau_1,\tau_2,\ldots,\tau_t$ and hence coincide on the part of the polygon $\mathcal{P}_{k+1}$ separated from the distinguished vertex by the diagonal $\widetilde{d}_{i_t}$. Since the part of the triangulation $\D(\textbf{n}_t)$ remaining after cutting  along $\widetilde{d}_{i_t}$ is precisely the initial triangulation $\D(\textbf{u}_{k-t})$ of $(k-t+1)$-gon $\mathcal{P}_{k-t+1}$, it follow immediately that each component of $\textbf{n}_t$, possibly with the exception
of the  two components corresponding to the boundary vertices of of the diagonal $\widetilde{d}_{i_t}$,  is less than or equal to
the corresponding component of the $k$-tuple $\textbf{b}$, since each component of $\textbf{n}$ is less than or equal to the
corresponding component of $\textbf{b}$ and each component of $\textbf{b}$ is at least $2$. To see that the two components of $\textbf{n}_t$ corresponding to the boundary vertices of $\widetilde{d}_{i_t}$ are also less than or equal to the corresponding components of $\textbf{b}$, we argue as follows:
Let $\D(\textbf{n}^t)$ denote the triangulation of the $(k-t+1)$-gon $\mathcal{P}_{k-t+1}$ corresponding
to the $(k-t)$-tuple $\textbf{n}^t$. This is a subtriangulation of $\D(\textbf{n})$ given by cutting along the diagonal $\widetilde{d}_{i_t}$.
Note that if either of the two components of $\textbf{n}^t$ corresponding to the boundary vertices of $\widetilde{d}_{i_t}$ is an interior component, then it is greater than $1$, since otherwise $\textbf{n}^t$ would still have $l$ interior $1$'s, contrary to assumption.
It follows that each of the components of $\textbf{n}_t$ corresponding to the boundary vertices of $\widetilde{d}_{i_t}$ is less than or equal
to the corresponding component of $\textbf{n}^t$ and hence less than or equal to the corresponding component of $\textbf{b}$.
\end{proof}

\begin{Ex} {\em We illustrate the proof of Lemma~\ref{lem: same} for $\textbf{n} = (3,1,4,3,1,2,4,1,4) \in \mathcal{Z}_9$. It is clear that $\dpt(\textbf{n})=3$ by definition.  By blowing down $\textbf{n}$ sequentially at the middle $\blue{1}$ three times, we obtain the sequence $$ \textbf{n}=(3,1,4,3,\blue{1},2,4,1,4) \to \textbf{n}^1 = (3,1,4,2,\blue{1},4,1,4) \to $$ $$\textbf{n}^2 = (3,1,4,\blue{1},3,1,4) \to \textbf{n}^3 = (3,1,3,2,1,4).$$ We stopped at $\textbf{n}^3$, since  $\dpt(\textbf{n}^3)=2$. The corresponding path $$ \D(\textbf{u}_9) =\D((1,2,2,2,2,2,2,2,1)),  \D(\textbf{n}_1) = \D((1,2,2,3,1,3,2,2,1)),$$ $$\D(\textbf{n}_2) = \D((1,2,2,4,1,2,3,2,1)),  \D(\textbf{n}_3) = \D((1,2,3,3,1,2,4,2,1)),$$ in $\mathcal{G}_9$ can be obtained as discussed in the proof of Lemma~\ref{lem: once}. In Figure~\ref{fig: decagon}, we depicted the triangulations $\D(\textbf{n}_3)$ (on the left) and $\D(\textbf{n})$ (on the right) of the decagon. We also highlighted the vertices of the diagonal $\widetilde{d}_3$ in the triangulation $\D(\textbf{n}_3)$, which play a crucial role in our proof of Lemma~\ref{lem: same}.

\begin{figure}[ht]  \relabelbox \small {\epsfxsize=5in
\centerline{\epsfbox{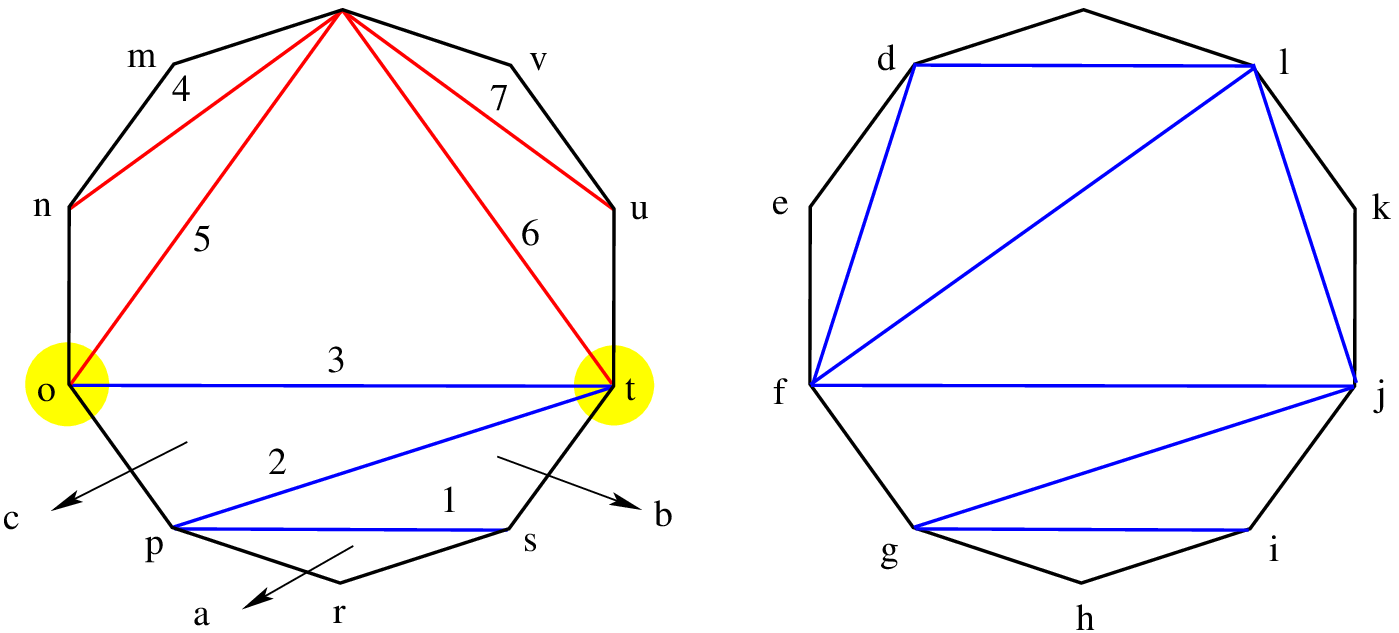}}}
\relabel{1}{\blue{$\widetilde{d}_4$}} \relabel{2}{\blue{$\widetilde{d}_5$}} \relabel{3}{\blue{$\widetilde{d}_3$}}  \relabel{4}{\red{$d_1$}} \relabel{5}{\red{$d_2$}}
  \relabel{6}{\red{$d_6$}}  \relabel{7}{\red{$d_7$}} \relabel{a}{$\tau_1$} \relabel{b}{$\tau_2$} \relabel{c}{$\tau_3$}   \relabel{d}{$3$}  \relabel{e}{$1$}  \relabel{f}{$4$}  \relabel{g}{$3$}  \relabel{h}{$1$}  \relabel{i}{$2$}  \relabel{j}{$4$}  \relabel{k}{$1$}  \relabel{l}{$4$}    \relabel{m}{$1$}  \relabel{n}{$2$}  \relabel{o}{$3$}  \relabel{p}{$3$}  \relabel{r}{$1$}  \relabel{s}{$2$}  \relabel{t}{$4$}  \relabel{u}{$2$}  \relabel{v}{$1$}
   \endrelabelbox
\caption{Triangulations $\D(\textbf{n}_3)$ and $\D(\textbf{n})$ of the decagon.}\label{fig: decagon} \end{figure}
When we cut the triangulation $\D(\textbf{n}_3)$ along the diagonal $\widetilde{d}_3$, or equivalently peel away the triangles $\tau_1, \tau_2, \tau_3$, the remaining subtriangulation (on the side of the distinguished vertex) is identical to the initial triangulation of $\mathcal{P}_{6}$.  It follows that if $\textbf{n} \leq \textbf{b}$ for some $k$-tuple $\textbf{b}$, with $b_i \geq 2$ for each $1 \leq i \leq k$, then $\textbf{n}_3  \leq \textbf{b}$ as well. Note that there is a unique pair $(p,q)$ of coprime integers with $p >q \geq 1$, so that $\frac{p}{p-q}=[b_1, b_2, \ldots, b_k]$. So, in other words, if $\textbf{n}$ belongs to $\mathcal{Z}_{k}(\textstyle{\frac{p}{p-q}})$, for some coprime pair $(p,q)$, so does $\textbf{n}_3$. We would like to emphasize that $\textbf{n} \leq \textbf{b}$ does not necessarily imply that $\textbf{n}_1 \leq \textbf{b}$, because of the fact that sixth component of $\textbf{n}_1$ is one higher than that of $\textbf{n}$, and it does not necessarily imply that $\textbf{n}_2 \leq \textbf{b}$, because of the fact that fourth component of $\textbf{n}_2$ is one higher than that of $\textbf{n}$.}
\end{Ex}

\begin{proof}[Proof of  Proposition~\ref{prop: depth}.]
Let $W_{p,q}(\textbf{n})$ denote the minimal symplectic filling of $(L(p,q), \xi_{can})$ that corresponds to $\textbf{n}= (n_1, \ldots, n_k) \in \mathcal{Z}_{k}(\textstyle{\frac{p}{p-q}})$, and let $\D(\textbf{n})$ be the corresponding vertex in the graph $\mathcal{G}_k^{p,q}$.  We now show that the rational blowdown depth  of  $W_{p,q}(\textbf{n})$  is bounded above by $\dpt (\textbf{n})$. The proof will be by induction on $r=\dpt (\textbf{n})$. First suppose that $r=0$, then $W_{p,q}(\textbf{n}) = W_{p,q}(\textbf{u}_k)$ is the canonical symplectic filling.

Now suppose that for some $l\geq 1$ the result is true for $0\leq r<l$. We show that the result remains true for $r=l$.
Suppose that  $\dpt (\textbf{n})=l$. Next, by setting $\textbf{n}^0=\textbf{n}$, and proceeding exactly as in proof of Lemma~\ref{lem: same}, we obtain the path
$$ \D(\textbf{u}_k)=\D(\textbf{n}_0),\D(\textbf{n}_1),\ldots,\D(\textbf{n}_t)$$ of triangulations and the associated triangles $\tau_1,\tau_2,\ldots,\tau_t$ contained in both $\D(\textbf{n}_t)$ and $\D(\textbf{n})$,  where $\dpt (\textbf{n}_t)=l-1$.
Furthermore, $\textbf{n}_t$ also belongs to $\mathcal{Z}_{k}(\textstyle{\frac{p}{p-q}})$, and hence $W_{p,q}(\textbf{n}_t)$ is also a minimal symplectic filling of $(L(p,q), \xi_{can})$.
Since the corresponding sequence of distinguished diagonals $d_{i_1},d_{i_2},\ldots,d_{i_t}$ is  contiguous,
 it follows from
the arguments given in the proof of Theorem~\ref{thm: main} that the monodromy of the planar Lefchetz fibration on $W_{p,q}(\textbf{n}_t)$
is obtained from the monodromy of the planar Lefschetz fibration on $W_{p,q}(\textbf{u}_k)$ by a {\em single} rational blowdown.

Now peel away the triangles $\tau_1,\tau_2,\ldots,\tau_t$ from $\D(\textbf{n})$ and $\D(\textbf{n}_t)$ to obtain triangulations
$\D(\overline{\textbf{n}})$ and $\D(\overline{\textbf{n}}_t)$ of $\mathcal{P}_{k-t+1}$. Note that  $\overline{\textbf{n}}$ will have $l-1$ interior $1$'s
and $\overline{\textbf{n}}_t=\textbf{u}_{k-t}$. By the induction hypothesis, it follows that the planar Lefschetz fibration on $W(\overline{\textbf{n}},\overline{\textbf{m}})$ is obtained from the planar Lefschetz fibration on
$W(\textbf{u}_{k-t},\overline{\textbf{n}}+\overline{\textbf{m}}-\textbf{u}_{k-t})$, where $\overline{\textbf{m}}$ is the
$(k-t)$-tuple having $1$ in each component that $\overline{\textbf{n}}$ has an interior $1$ and $0$ elsewhere, by at
most $l-1$ rational blowdowns. It follows that $W_{p,q}(\textbf{n})$ can be obtained from $W_{p,q}(\textbf{n}_t)$ by at most
$l-1$ rational blowdowns and hence from the canonical symplectic filling $W_{p,q}(\textbf{u}_k)$ by at most $l$ rational blowdowns.
\end{proof}

\begin{Rem} \label{rem: path} {\em Proposition~\ref{prop: depth} shows that, for each $\D(\textbf{n}) \in \mathcal{G}^{p,q}_k$,  there is a path of length $\dpt (\textbf{n})$ from $\D(\textbf{u}_k)$ to $\D(\textbf{n})$ in  $\mathcal{G}^{p,q}_k$. On the other hand, Lemma~\ref{lem: dep} implies that the minimum length of a path from $\D(\textbf{u}_k)$ to $\D(\textbf{n})$ in  $\mathcal{G}^{p,q}_k$, is at least $\dpt (\textbf{n})$. Therefore, $\dpt (\textbf{n})$ is equal to the "path-length" of $\D(\textbf{n})$ in  $\mathcal{G}^{p,q}_k$, which is defined to be the length of the shortest directed path in $\mathcal{G}^{p,q}_k$, starting from the root vertex $\D(\textbf{u}_k)$ and ending at the vertex $\D(\textbf{n})$.  } \end{Rem}

\noindent {\bf Executive summary:} Here we explicitly describe the rational blowdown algorithm. Suppose that $ \textbf{n} \in \mathcal{Z}_{k}(\textstyle{\frac{p}{p-q}})$ for some coprime pair $(p,q)$ with $1 \leq q < p$ such that $\dpt(\textbf{n})=l$. We enumerate the interior $1$'s in the $k$-tuple $\textbf{n}$ from left to right and blowdown repeatedly the interior $1$ of the same enumeration (for instance, leftmost) until we get to some $(k-t_1)$-tuple $\textbf{n}^{t_1}$ with  $\dpt(\textbf{n}^{t_1})=l-1$. We repeat the same process to obtain a sequence $\textbf{n}, \textbf{n}^{t_1}, \ldots \textbf{n}^{t_l}$ so that $\dpt(\textbf{n}^{t_j})=l-j$. Then there is a corresponding path
 $\D(\textbf{u}_k), \D(\textbf{n}_{t_1}), \ldots, \D(\textbf{n}_{t_l})= \D(\textbf{n})$ in $\mathcal{G}^{p,q}_k$ so that the minimal symplectic filling $W_{p,q}(\textbf{n}_{t_{j+1}})$ can be obtained from the minimal symplectic filling $W_{p,q}(\textbf{n}_{t_{j}})$  by a single rational blowdown along a linear plumbing graph.  Note that the triangulation $\D(\textbf{n}_{t_{j+1}})$ is obtained from the triangulation $\D(\textbf{n}_{t_{j}})$ by diagonal flips along a {\em contiguous} sequence of distinguished diagonals in  $\D(\textbf{n}_{t_{j}})$. Therefore, the linear plumbing graph for this rational blowdown can be obtained  using Proposition~\ref{prop: lpg}.

\begin{Ex} {\em Let  $\textbf{n} = (3,1,4,3,1,2,4,1,4) \in \mathcal{Z}_9$. Note  that $\dpt(\textbf{n})=3$. Fix a coprime pair $(p,q)$ with $1 \leq q < p$ such that the length of the Hirzebruch-Jung continued fraction expansion of $\textstyle{\frac{p}{p-q}}$ is $9$ and $ \textbf{n} \in \mathcal{Z}_{9}(\textstyle{\frac{p}{p-q}})$. By blowing down $\textbf{n}$ sequentially at the middle interior $1$ three times, we obtain $\textbf{n}^{t_1} = (3,1,3,2,1,4)$,  where $\dpt(\textbf{n}^{t_1})=2$. By blowing down $\textbf{n}^{t_1}$  sequentially at the rightmost interior $1$ twice, we obtain $\textbf{n}^{t_2} = (3,1,2,2)$,  where $\dpt(\textbf{n}^{t_2})=1$. By blowing down $\textbf{n}^{t_2}$  sequentially at the  unique interior $1$ twice, we obtain $\textbf{n}^{t_3} = (1,1)$,  where $\dpt(\textbf{n}^{t_3})=0$. The corresponding  path in $\mathcal{G}^{p,q}_9$ is given by
 $\D(\textbf{u}_9), \D(\textbf{n}_{t_1}) = \D((1,2,3,3,1,2,4,2,1))$, $\D(\textbf{n}_{t_2})= \D((1,2,4,3,1,2,4,1,3))$,  $\D(\textbf{n}) = \D(\textbf{n}_{t_3}) =\D((3,1,4,3,1,2,4,1,4))$.

 As a matter of fact $ \D(\textbf{n})$ is obtained from $\D(\textbf{u}_9)$ by diagonal flips along the sequence $d_4$, $d_5$, $d_3$, $d_7$, $d_6$, $d_1$, $d_2$ of distinguished diagonals in $\D(\textbf{u}_9)$, which indeed determines a path in $\mathcal{G}_{9}$. Our algorithm partitions this sequence as follows: the sequence $d_4$, $d_5$, $d_3$ of distinguished diagonals is  contiguous in $\D(\textbf{u}_9)$, the sequence of distinguished diagonals $d_7$, $d_6$ is contiguous in $\D(\textbf{n}_{t_1})$, and the sequence $d_1$, $d_2$ of distinguished diagonals is contiguous in $\D(\textbf{n}_{t_2})$. We conclude that the minimal symplectic filling $W_{p,q}(\textbf{n}_{t_{1}})$ can be obtained from the canonical symplectic filling $W_{p,q}(\textbf{u}_9)$ by a rational blowdown along the linear plumbing graph with weights $-3,-5,-2$, the minimal symplectic filling $W_{p,q}(\textbf{n}_{t_{2}})$ can be obtained from $W_{p,q}(\textbf{n}_{t_{1}})$ by a rational blowdown along the linear plumbing graph with weights $-5,-2$, and finally $W_{p,q}(\textbf{n})$ can be obtained from $W_{p,q}(\textbf{n}_{t_{2}})$ by a rational blowdown along the linear plumbing graph with weights $-2,-5$. }
\end{Ex}

\section{Examples of rational blowdown graphs}

In this section, to avoid cumbersome notation, once a pair $(p,q)$ of coprime integers with $p > q \geq 1$ is fixed, for any $k$-tuple $\textbf{n}=(n_1, \ldots, n_k) \in \mathcal{Z}_{k}$ we will speak about the $4$-manifold $\textbf{n}$ referring to $W_{p,q} (\textbf{n})$ as in Definition~\ref{def: achiral}. If $\textbf{n} \in \mathcal{Z}_{k}(\textstyle{\frac{p}{p-q}})$, we will refer to $\textbf{n}$ as a minimal symplectic filling of $(L(p, q), \xi_{can})$. We will also speak about a rational blowdown $\textbf{n} \to \textbf{n}'$, for a pair of minimal symplectic fillings $\textbf{n},  \textbf{n}' \in \mathcal{Z}_{k}(\textstyle{\frac{p}{p-q}})$. Moreover, by the triangulation $\textbf{n}$, we will mean $\D(\textbf{n})$ as in Definition~\ref{def: inverse}.

In the following, for $2 \leq s$,  we denote the curve $\d_{s,s}$ in Definition~\ref{def: curves} by $\a_s$. For the definition of the curves $\b_{i,t,j}$ in the disk $D_k$ with $k$-holes,  we refer to Figure~\ref{fig: beta}.

\subsection{Example A} \label{ex: A} Let $(p,q)=(24,7)$. Then $\frac{p}{p-q}=[2,2,4,2,2]$. The contact $3$-manifold $(L(24, 7), \xi_{can})$ has $4$ distinct minimal symplectic fillings, up to diffeomorphism, which are parametrized by the set $$\mathcal{Z}_{5}(\textstyle{\frac{24}{17}}) =\{(1,2,2,2,1), (2,1,3,2,1), (1,2,3,1,2), (2,1,4,1,2)\}.$$ The set of vertices of the graded, directed, rooted, connected graph $\mathcal{G}^{24, 7}_5$ consists of the $4$ triangulations of the hexagon which belongs to the  set $\mathcal{T}^{24,7} (P_6) \subset \mathcal{T}(P_6)$. These triangulations, each of which represents a distinct minimal symplectic filling of $(L(24, 7), \xi_{can}),$ are  encircled in red in  Figure~\ref{fig: graph247}.

\begin{figure}[ht] \small {\epsfxsize=5in
\centerline{\epsfbox{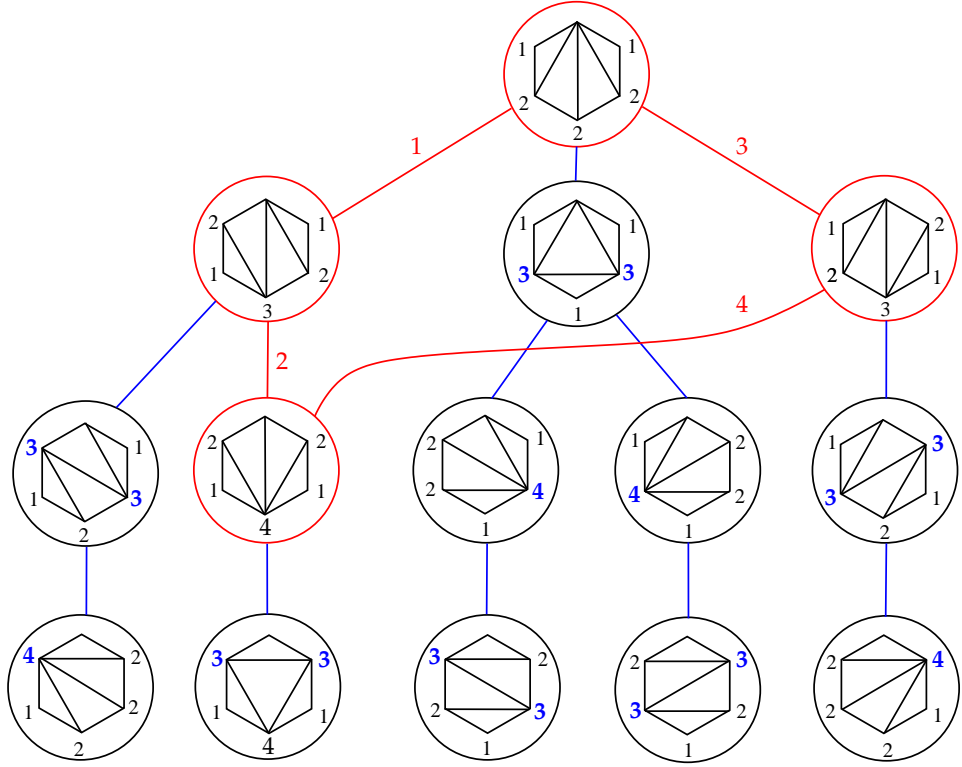}}}
\caption{The set of vertices of the  graded, directed, rooted,  connected graph $\mathcal{G}^{24, 7}_5$ consists of the triangulations of the hexagon encircled in red, each of which represents a distinct minimal symplectic filling of $(L(24, 7), \xi_{can})$, and the edges are the red arcs, each of which represents a rational blowdown.  }\label{fig: graph247} \end{figure}

There are $4$ edges of the graph $\mathcal{G}^{24, 7}_5$, consisting of the red arcs labelled by $1$, $2$, $3$ and $4$ in  Figure~\ref{fig: graph247}, corresponding to the edges $e_1$, $e_{1,3}$, $e_{3}$, and $e_{3,1}$, respectively,  in the graph $\mathcal{G}_5$. We denote the rational blowdowns represented by these arcs as $\rbd_1$,  $\rbd_2$,  $\rbd_3$,  and $\rbd_4$, respectively. By definition,  each edge of $\mathcal{G}^{24, 7}_5$ is given by the concatenation of some directed edges of $\mathcal{G}_5$.   Note that  there is no edge  in $\mathcal{G}^{24, 7}_5$ from the root vertex $(1,2,2,2,1)$ to the vertex $(2,1,4,1,2)$, since there are two distinct paths, $e_1 \sbullet[.75] e_{1,3}$ and $e_3 \sbullet[.75] e_{3,1},$ from $(1,2,2,2,1)$ to $(2,1,4,1,2)$ in  $\mathcal{G}_5$, where $\sbullet[.75]$ denotes the concatenation of the edges.

In the following, using the algorithm  described  in Section~\ref{sec: lantern}, we will explicitly describe the lantern substitution corresponding to each of the rational blowdowns in Figure~\ref{fig: graph247}. First of all,  we observe that the monodromy of the planar Lefschetz fibration on the minimal resolution $(1,2,2,2,1)$ is the product $$   D(\a_2) \circ D(\a_3) \circ D(\a_4)  \circ D(\a_5) \circ D(\g_1) \circ D^2(\g_3) \circ D(\g_5)$$ of Dehn twists along the curves given in Figure~\ref{fig: monod247}.

 \begin{figure}[ht]  \relabelbox \small {\epsfxsize=4.5in
\centerline{\epsfbox{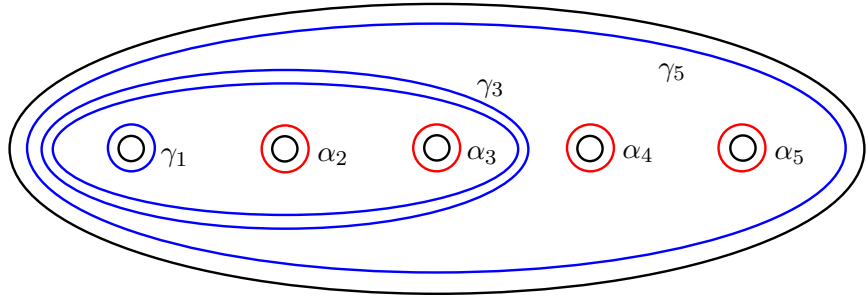}}}
\relabel{1}{$\g_1$} \relabel{2}{$\a_2$} \relabel{3}{$\a_3$} \relabel{4}{$\a_4$} \relabel{5}{$\a_5$} \relabel{6}{$\g_3$}  \relabel{8}{$\g_5$}
\endrelabelbox
\caption{Initial monodromy curves in the disk $D_5$ with $5$-holes.}\label{fig: monod247} \end{figure}

The aforementioned product of Dehn twists is also the monodromy of the planar open book $\OB_{24,7}$ that supports the contact $3$-manifold $(L(24,7), \xi_{can})$. Note that {\em each} triangulation  in  Figure~\ref{fig: graph247}, represents a smooth $4$-manifold  with boundary $L(24,7)$, together with a planar Lefschetz fibration whose monodromy is a different factorization, positive for the ones encircled in red, and achiral otherwise, of the initial monodromy of the planar open book $\OB_{24,7}$.

The red arc labelled by $1$ in Figure~\ref{fig: graph247}  represents $\rbd_1$, that corresponds to a lantern substitution in the monodromy of the planar Lefschetz fibration on the canonical symplectic filling $(1,2,2,2,1)$, as follows. We encircled the components of the $5$-tuple  which increase and boxed the ones that decrease along the edge $e_1$ below, which is completely determined by the diagonal flip along $d_1$:  $$(1,2,2,2,1) \xrightarrow[\text{lantern}]{\text{$\rbd_1$}}(\circled{2},\boxed{1},\circled{3},2,1).$$ Thus, to realize $\rbd_1$, we apply a lantern substitution along the $4$-holed sphere in $D_5$ bounded by the curves $\g_1, \a_2, \a_3, \g_3$ (see Figure~\ref{fig: monod247}), where we replace $$  D(\g_1) \circ D(\a_2) \circ D(\a_3) \circ  D(\g_3)$$ by $$  D(\d_{2,3}) \circ D(\b_{1,2,3}) \circ D(\g_2).$$  It follows that the  monodromy of the planar Lefschetz fibration on the minimal symplectic filling $(2,1,3,2,1)$ is given by $$  D(\d_{2,3}) \circ D(\b_{1,2,3}) \circ D(\g_2)  \circ D(\g_3) \circ D(\a_4) \circ D(\a_5) \circ  D(\g_5).$$

The red arc labelled by $2$ in Figure~\ref{fig: graph247},  which represents $\rbd_{2}$, corresponds to a lantern substitution
in the monodromy of the planar Lefschetz fibration on the minimal symplectic filling $(2,1,3,2,1)$, as follows. Based on the diagonal flip $$(2,1,3,2,1) \xrightarrow[\text{lantern}]{\text{$\rbd_2$}}(2,1,\circled{4},\boxed{1},\circled{2}),$$ along $d_3$, we apply a lantern substitution along the $4$-holed sphere in $D_5$ bounded by the curves $\g_3, \a_4, \a_5, \g_5$ (see Figure~\ref{fig: monod247}), where we replace $$  D(\g_3) \circ D(\a_4) \circ D(\a_5) \circ  D(\g_5)$$ by $$  D(\d_{4,5}) \circ D(\b_{3,4,5}) \circ D(\g_4).$$   As a result the  monodromy of the planar Lefschetz fibration on the minimal symplectic filling $(2,1,4,1,2)$ is given by
\begin{equation}\label{eq1}
\begin{aligned}
 & D(\d_{2,3}) \circ D(\b_{1,2,3}) \circ D(\g_2)   \circ D(\d_{4,5}) \circ D(\b_{3,4,5}) \circ D(\g_4) = \\
& D(\d_{2,3}) \circ D(\d_{4,5}) \circ D(\b_{3,4,5})\circ D(\b_{1,2,3}) \circ D(\g_2) \circ D(\g_4). \end{aligned} \end{equation}

The red arc labelled by $3$ in Figure~\ref{fig: graph247}, which  represents $\rbd_3$,  corresponds to a lantern substitution in the monodromy of the planar Lefschetz fibration on the canonical symplectic filling $(1,2,2,2,1)$, as follows. Based on the diagonal flip $$(1,2,2,2,1) \xrightarrow[\text{lantern}]{\text{$\rbd_3$}}(1,2,\circled{3},\boxed{1},\circled{2}),$$ along $d_3$, we apply a lantern substitution along the $4$-holed sphere in $D_5$ bounded by the curves $\g_3, \a_4, \a_5, \g_5$ (see Figure~\ref{fig: monod247}), where we replace $$  D(\g_3) \circ D(\a_4) \circ D(\a_5) \circ  D(\g_5)$$ by $$  D(\d_{4,5}) \circ D(\b_{3,4,5}) \circ D(\g_4).$$  It follows that the  monodromy of the planar Lefschetz fibration on the minimal symplectic filling $(1,2,3,1,2)$, is given by $$ D(\g_1)  \circ D(\a_2) \circ D(\a_3) \circ D(\g_3)\circ  D(\d_{4,5}) \circ D(\b_{3,4,5}) \circ D(\g_4).$$

 The red arc labelled by $4$ in Figure~\ref{fig: graph247}, which represents $\rbd_{4}$, corresponds to a lantern substitution in the monodromy of the planar Lefschetz fibration on the minimal symplectic filling  $(1,2,3,1,2)$,  as follows. Based on the diagonal flip $$(1,2,3,1,2) \xrightarrow[\text{lantern}]{\text{$\rbd_{4}$}}(\circled{2}, \boxed{1},\circled{4},1,2),$$ along $d_1$,  we apply a lantern substitution along the $4$-holed sphere in $D_5$ bounded by the curves $\g_1, \a_2, \a_3, \g_3$ (see Figure~\ref{fig: monod247}), where we replace $$  D(\g_1) \circ D(\a_2) \circ D(\a_3) \circ  D(\g_3)$$ by $$  D(\d_{2,3}) \circ D(\b_{1,2,3}) \circ D(\g_2).$$   As a result the  monodromy of the planar Lefschetz fibration on the minimal symplectic filling $(2,1,4,1,2)$ is given by \begin{equation}\label{eq2}
\begin{aligned}
 &  D(\d_{2,3}) \circ D(\b_{1,2,3}) \circ D(\g_2)   \circ D(\d_{4,5}) \circ D(\b_{3,4,5}) \circ D(\g_4) = \\
&  D(\d_{2,3}) \circ D(\d_{4,5}) \circ D(\b_{3,4,5})\circ D(\b_{1,2,3}) \circ D(\g_2) \circ D(\g_4), \end{aligned} \end{equation} which indeed agrees with the factorization~\eqref{eq1} above.

Note that the concatenation of $\rbd_1$ and $\rbd_2$ (and equivalently, the concatenation of $\rbd_3$ and $\rbd_4$) can be viewed as a monodromy substitution, where the initial factorization $$ D(\a_2) \circ D(\a_3) \circ D(\a_4)  \circ D(\a_5) \circ D(\g_1) \circ D^2(\g_3) \circ D(\g_5)$$ is replaced with the product $$D(\d_{2,3}) \circ D(\d_{4,5}) \circ D(\b_{3,4,5})\circ D(\b_{1,2,3}) \circ D(\g_2) \circ D(\g_4).$$ However, this monodromy substitution does not correspond to a rational blowdown, since otherwise  $(L(24, 7), \xi_{can})$ would have a rational homology ball filling, which contradicts Proposition~\ref{prop: wahl}. As a matter of fact, the minimal symplectic filling $(2,1,4,1,2)$ cannot be obtained from the canonical symplectic filling $(1,2,2,2,1)$, by a single rational blowdown as we show in Proposition~\ref{prop: depth2}.

\begin{Prop}\label{prop: depth2} The minimal symplectic filling $(2,1,4,1,2)$ of $(L(24, 7), \xi_{can})$ cannot be obtained from the canonical symplectic filling $(1,2,2,2,1)$, by a single rational blowdown. Consequently, the rational blowdown depth of the minimal symplectic filling $(2,1,4,1,2)$ is equal to $\dpt((2,1,4,1,2))=2$.
\end{Prop}

\begin{proof} We first  observe that  the canonical symplectic filling $\textbf{u}_5=(1,2,2,2,1)$ of $(L(24, 7), \xi_{can})$ is diffeomorphic to the linear plumbing of disk bundles over a sphere with weights $-4,-2,-4$ because $\frac{24}{7} =[4,2,4]$. It follows that the lattice $H_2(\textbf{u}_5, \mathbb{Z})$ is even, i.e. the square of any class in $H_2(\textbf{u}_5, \mathbb{Z})$  is an even integer. Note that since $(2,1,4,1,2)$ has height $2$, the linear plumbing graph for a possible rational blowdown from $\textbf{u}_5$ to $(2,1,4,1,2)$ must have  two vertices with weights $-2$ and $-5$. However,  $\textbf{u}_5$ cannot contain an embedded sphere of {\em odd} self-intersection, since  $H_2(\textbf{u}_5, \mathbb{Z})$ is even.
\end{proof}

\subsection{Example B} \label{ex: B} Let $(p,q)=(81,47)$. Then $\frac{p}{p-q}=[3,2,3,3,3]$. The contact $3$-manifold $(L(81, 47), \xi_{can})$ has $6$ distinct minimal symplectic fillings, up to diffeomorphism, which are parametrized by the set $\mathcal{Z}_{5}(\textstyle{\frac{81}{81-47}}) = $ $$ \{(1,2,2,2,1), (2,1,3,2,1), (1,2,3,1,2), (3,1,2,3,1),(3,1,3,1,3), (3,2,1,3,2)\}.$$ The set of vertices of the graded, directed, rooted, connected graph $\mathcal{G}^{81, 47}_5$ consists of the $6$ triangulations of the hexagon which belongs to the  set $\mathcal{T}^{81,47} (P_6) \subset \mathcal{T}(P_6)$. These triangulations are  encircled in red in  Figure~\ref{fig: graph8147}, and each one of them represents a distinct minimal symplectic filling of $(L(81, 47), \xi_{can})$.

\begin{figure}[ht] \small {\epsfxsize=5in
\centerline{\epsfbox{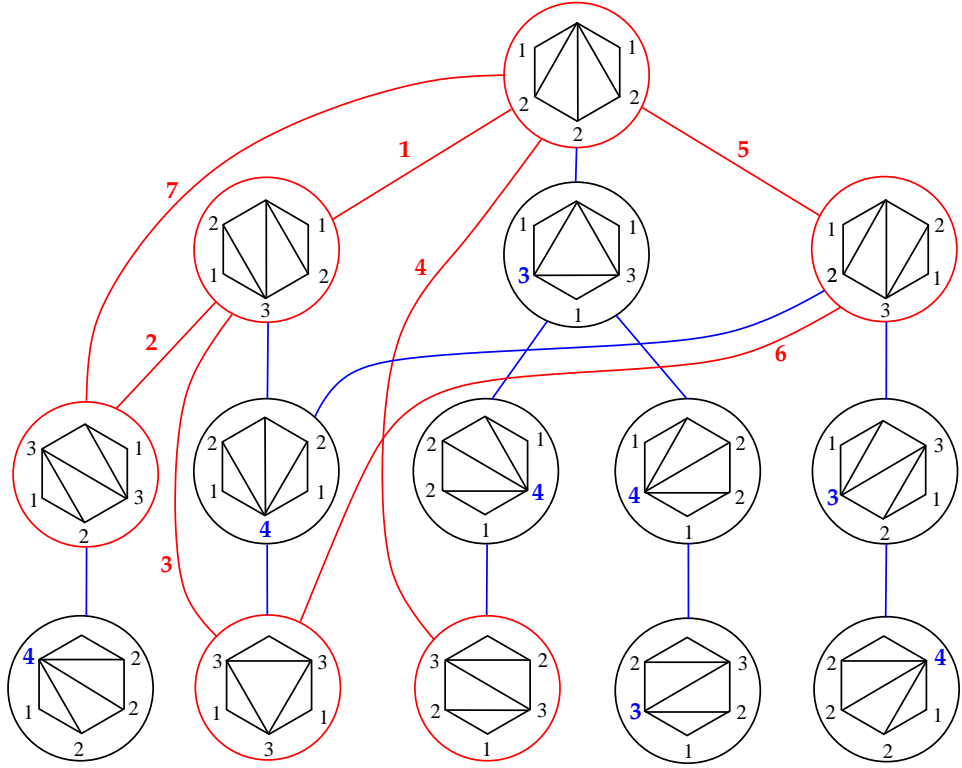}}}
\caption{The set of vertices of the  graded, directed, rooted,  connected graph $\mathcal{G}^{81, 47}_5$ consists of the triangulations of the hexagon encircled in red, each of which represents a distinct minimal symplectic filling of $(L(81, 47), \xi_{can})$, and the edges are the red arcs, each of which represents a rational blowdown.  }\label{fig: graph8147} \end{figure}

There are $7$ edges of the graph $\mathcal{G}^{81, 47}_5$, consisting of the red arcs enumerated from $1$ to $7$  in  Figure~\ref{fig: graph8147}, each of which represents a rational blowdown. By definition,  each edge of $\mathcal{G}^{81, 47}_5$ is given by the concatenation of some directed edges of  $\mathcal{G}_5$ as follows, $\rbd_1 = e_1,\;$ $\rbd_2 = e_{1,2}, \;$ $\rbd_3 = e_{1,3} \sbullet[.75]  e_{1,3,2},\; $ $\rbd_4 = e_2 \sbullet[.75]  e_{2,1} \sbullet[.75]  e_{2,1,3},\; $ $ \rbd_5 = e_3, \;$ $\rbd_6 = e_{3,1} \sbullet[.75]   e_{3,1,2}$, and $\rbd_7 =  e_1 \sbullet[.75] e_{1,2}$. Note that  the paths $e_1 \sbullet[.75] e_{1,3} \sbullet[.75] e_{1,3,2}$ and  $e_3 \sbullet[.75] e_{3,1} \sbullet[.75] e_{3,1,2}$
 in $\mathcal{G}_5$ are not edges in $\mathcal{G}^{81, 47}_5$, since there are two distinct paths from $(1,2,2,2,1)$ to $(3,1,3,1,3)$ in  $\mathcal{G}_5$.

In the following, we will explicitly describe the lantern substitutions needed for each of the rational blowdowns in Figure~\ref{fig: graph8147}. First of all,  we observe that the initial monodromy corresponding to the planar Lefschetz fibration on the minimal resolution $(1,2,2,2,1)$ is the product $$   D(\a_2) \circ D(\a_3) \circ D(\a_4)  \circ D(\a_5) \circ D^2(\g_1) \circ D(\g_3)  \circ D(\g_4)  \circ D^2(\g_5)$$ of Dehn twists along the curves given in Figure~\ref{fig: monod}.

 \begin{figure}[ht]  \relabelbox \small {\epsfxsize=4.5in
\centerline{\epsfbox{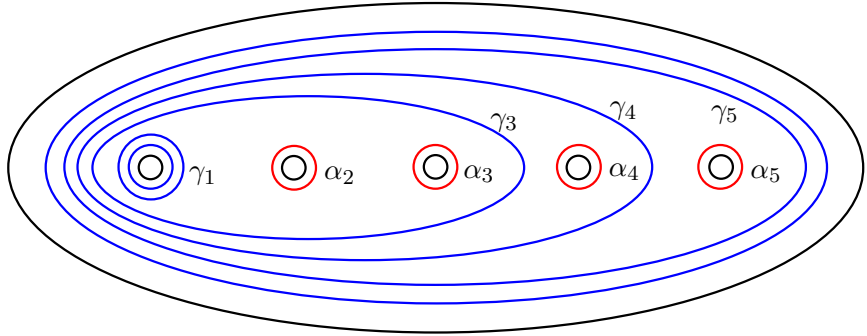}}}
\relabel{1}{$\g_1$} \relabel{2}{$\a_2$} \relabel{3}{$\a_3$} \relabel{4}{$\a_4$} \relabel{5}{$\a_5$} \relabel{6}{$\g_3$} \relabel{7}{$\g_4$} \relabel{8}{$\g_5$}
\endrelabelbox
\caption{Initial monodromy curves in the disk $D_5$ with $5$-holes.}\label{fig: monod} \end{figure}

The aforementioned product of Dehn twists is also the monodromy of the planar open book $\OB_{81,47}$ that supports the contact $3$-manifold $(L(81,47), \xi_{can})$. Note that {\em each} triangulation  in  Figure~\ref{fig: graph8147}, represents a smooth $4$-manifold  with boundary $L(81,47)$, together with a planar Lefschetz fibration whose monodromy is a different factorization, positive for the ones encircled in red, and achiral otherwise, of the initial monodromy of the planar open book $\OB_{81,47}$.

The red arc labelled by $1$ in Figure~\ref{fig: graph8147} (which is the edge $e_1$ in the graph $\mathcal{G}_5$), represents $\rbd_1$, that corresponds to a lantern substitution in the monodromy of the planar Lefschetz fibration on the canonical symplectic filling $(1,2,2,2,1)$, as follows. We encircled the entries of the $5$-tuple  which increase and boxed the ones that decrease along the edge $e_1$ below, which is completely determined by the diagonal flip along $d_1$:  $$(1,2,2,2,1) \xrightarrow[\text{lantern}]{\text{$e_1$}}(\circled{2},\boxed{1},\circled{3},2,1).$$ Thus, to realize $\rbd_1$, we apply a lantern substitution along the $4$-holed sphere in $D_5$ bounded by the curves $\g_1, \a_2, \a_3, \g_3$ (see Figure~\ref{fig: monod}). As a result of this lantern substitution, one of the Dehn twists that appear in the new factorization is along the convex curve $\d_{2,3}$.

The red arc labelled by $2$ in Figure~\ref{fig: graph8147} (which is the edge $e_{1,2}$ in the graph $\mathcal{G}_5$), represents $\rbd_2$, that corresponds to a lantern substitution in the monodromy of the planar Lefschetz fibration on the minimal symplectic filling $(2,1,3,2,1)$. The lantern substitution is determined by the diagonal flip along $d_2$:  $$(2,1,3,2,1) \xrightarrow[\text{lantern}]{\text{$e_{1,2}$}}(\circled{3},1 ,\boxed{2},\circled{3},1).$$ Thus, to realize $\rbd_2$, we apply a lantern substitution along the $4$-holed sphere bounded by the curves $\g_1, \d_{2,3}, \a_4, \g_4$ in $D^5$. Note that the Dehn twist along $\d_{2,3}$ appeared as a result of $\rbd_1$, and it belongs to the positive factorization of the monodromy of the planar Lefschetz fibration on the minimal symplectic filling $(2,1,3,2,1)$.

The red arc labelled by $7$ in Figure~\ref{fig: graph8147} (which is the concatenation $e_1  \sbullet[.75] e_{1,2}$ of the edges $e_1$ and $e_{1,2}$ in the graph $\mathcal{G}_5$), represents $\rbd_7$. In other words, the minimal symplectic filling $(3,1,2,3,1)$ can be obtained from  the minimal resolution $(1,2,2,2,1)$ by a {\em single} rational blowdown, that is a concatenation of  $\rbd_1$ and $\rbd_2$. Here is an explanation of this phenomenon using explicit monodromy substitutions. We observe that in the concatenation of the two lantern substitutions we use Dehn twists along the curves $\g_1, \g_1, \g_3, \g_4, \a_2, \a_3, \a_4$, all of which belong to the set of initial monodromy curves in Figure~\ref{fig: monod},  and  Dehn twists along the curves
$\g_2, \g_3$ (again belonging to the initial set of curves) emerge as a result of these substitutions. So, in the concatenation,  we could just use Dehn twists along the curves $\g_1, \g_1,  \g_4, \a_2, \a_3, \a_4$. This means that, there is a relation in the mapping class group of the $5$-holed sphere in $D_5$ bounded by the curves $\g_1, \a_2, \a_3, \a_4, \g_4$, where the product $$D^2(\g_1) \circ D(\a_2) \circ D(\a_3) \circ D(\a_4) \circ D(\g_4)$$ is isotopic to the product of Dehn twists about $4$ curves. One can easily verify that this relation is precisely the {\em daisy relation} (see \cite{emvhm}). Since the triangulation $(3,1,2,3,1)$ is obtained from $(1,2,2,2,1)$ by flips along the contiguous sequence $d_1, d_2$ of distinguished diagonals in $(1,2,2,2,1)$, the linear plumbing graph for $\rbd_7$ has weights  $-2,-5$, by Proposition~\ref{prop: lpg}. We would like to point out that the rational blowdown  $\rbd_7$  is not immediately visible in the dual plumbing graph of the corresponding cyclic quotient singularity, which is the linear graph with weights $-2,-4,-3,-3,-2$.

Note  that the minimal symplectic filling $(3,1,3,1,3)$  of height $3$ in Figure~\ref{fig: graph8147},   can be obtained from the minimal resolution $(1,2,2,2,1)$ via two distinct rational
blowdown sequences: Apply  $\rbd_1$ first and then $\rbd_3$ or apply $\rbd_5$ first and then $\rbd_6$. We already discussed the lantern substitution for $\rbd_1$ and the monodromy substitution for $\rbd_3$ can be seen as follows. First of all, $\rbd_3$ is obtained by a concatenation of two lantern substitutions corresponding to the edges $e_{1,2}$ and $e_{1,3,2}$ in the graph $\mathcal{G}_5$, respectively, as follows: $$(2,1,3,2,1) \xrightarrow[\text{lantern}]{\text{$e_{1,3}$}}(2,1,\circled{4},\boxed{1},\circled{2}) \xrightarrow[\text{lantern}]{\text{$e_{1,3,2}$}}  (\circled{3},1,\boxed{3},1,\circled{3}). $$

Starting from the monodromy factorization for $(2,1,3,2,1)$, we need to insert a cancelling pair of Dehn twists along $\g_3$, which is dictated by the fact that the triangulation $(2,1,4,1,2)$ that we skip  does not belong to $\mathcal{T}^{81,47} (P_6)$, exactly because its {\em third} entry (colored blue in Figure~\ref{fig: graph8147}) is one higher than the third entry of $(3,2,3,3,3)$.
 Then we apply a lantern substitution along the $4$-holed sphere bounded by the curves $\g_3, \a_4, \a_5, \g_5$.  It is clear that the resulting monodromy factorization has a negative Dehn twist along $\g_3$, and we obtain an {\em achiral} planar  Lefschetz fibration on $(2,1,4,1,2)$ as expected, since the triangulation $(2,1,4,1,2)$  is indeed included in the  graph $\mathcal{G}_5$ but not in  $\mathcal{G}^{81, 47}_5$!

 Nevertheless, as a result  of the first lantern substitution corresponding to the edge $e_{1,3}$,  a  Dehn twist emerges  in the new factorization  along the convex curve $\d_{4,5}$, which allows one to apply a lantern substitution along the $4$-holed sphere bounded by the curves $\g_1, \d_{2,3}, \d_{4,5}, \g_5$.   The resulting monodromy factorization of the planar  Lefschetz fibration on the minimal symplectic filling $(3,1,3,1,3)$ is
positive again since a Dehn twist along $\g_3$ emerges  in the new factorization to cancel out the negative one. Since the triangulation $(3,1,3,1,3)$ is obtained from $(2,1,3,2,1)$ by flips along the contiguous sequence $d_3, d_2$ of distinguished diagonals in $(2,1,3,2,1)$, the linear plumbing graph for $\rbd_3$ has weights  $-5,-2$, by Proposition~\ref{prop: lpg}.

We would like to show that $\rbd_3$ is a daisy substitution for the $5$-holed sphere and illustrate an important step in our proof of Theorem~\ref{thm: main}. Consider the path $$\textbf{n}_0 = (2,1,3,2,1),  \; \textbf{n}_1 = (2,1,4,1,2), \; \textbf{n}_2=(3,1,3,1,3)$$ in $\mathcal{G}_5$ and let $$\overline{\textbf{n}}_0 = (1,2,2,1), \;  \overline{\textbf{n}}_1 = (1,3,1,2), \;   \overline{\textbf{n}}_2 = (2,2,1,3)$$ be the corresponding path in $\mathcal{G}_4$, obtained by blowing down the $1$ in the second entry of each of the $5$-tuples $\textbf{n}_0, \textbf{n}_1, \textbf{n}_2$. Note that this corresponds, geometrically, to peeling off the same triangle from each one of the triangulations $\textbf{n}_0, \textbf{n}_1, \textbf{n}_2$. Now in the concatenation of the two lantern substitutions in $D_4$! corresponding to $$(1,2,2,1) \xrightarrow[\text{lantern}]{}(1,\circled{3},\boxed{1},\circled{2}) \xrightarrow[\text{lantern}]{}  (\circled{2},\boxed{2},1,\circled{3}), $$ we use Dehn twists along the curves $\overline{\g}_1, \overline{\g}_2, \overline{\g}_4,  \overline{\g}_4, \overline{\a}_2, \overline{\a}_3, \overline{\a}_4$ (overline is used to denote the curves in $D_4$, to distinguish them from the curves in $D_5$)   and Dehn twists along the curves $ \overline{\g}_2,  \overline{\g}_3$ emerge as a result of these substitutions. Overall, we substitute the product of Dehn twists along the curves $\overline{\g}_1, \overline{\g}_4,  \overline{\g}_4, \overline{\a}_2, \overline{\a}_3, \overline{\a}_4$, by Dehn twists along four curves. One can easily verify that this is nothing but a daisy substitution.  To see that, the path $\textbf{n}_0, \textbf{n}_1, \textbf{n}_2$ is also a rational blowdown of the same type, we just paste back the triangle we peeled away, which has the effect of embedding $D_4$ into $D_5$ as illustrated in Figure~\ref{fig: embed}.

 \begin{figure}[ht]  \relabelbox \small {\epsfxsize=3.2in
\centerline{\epsfbox{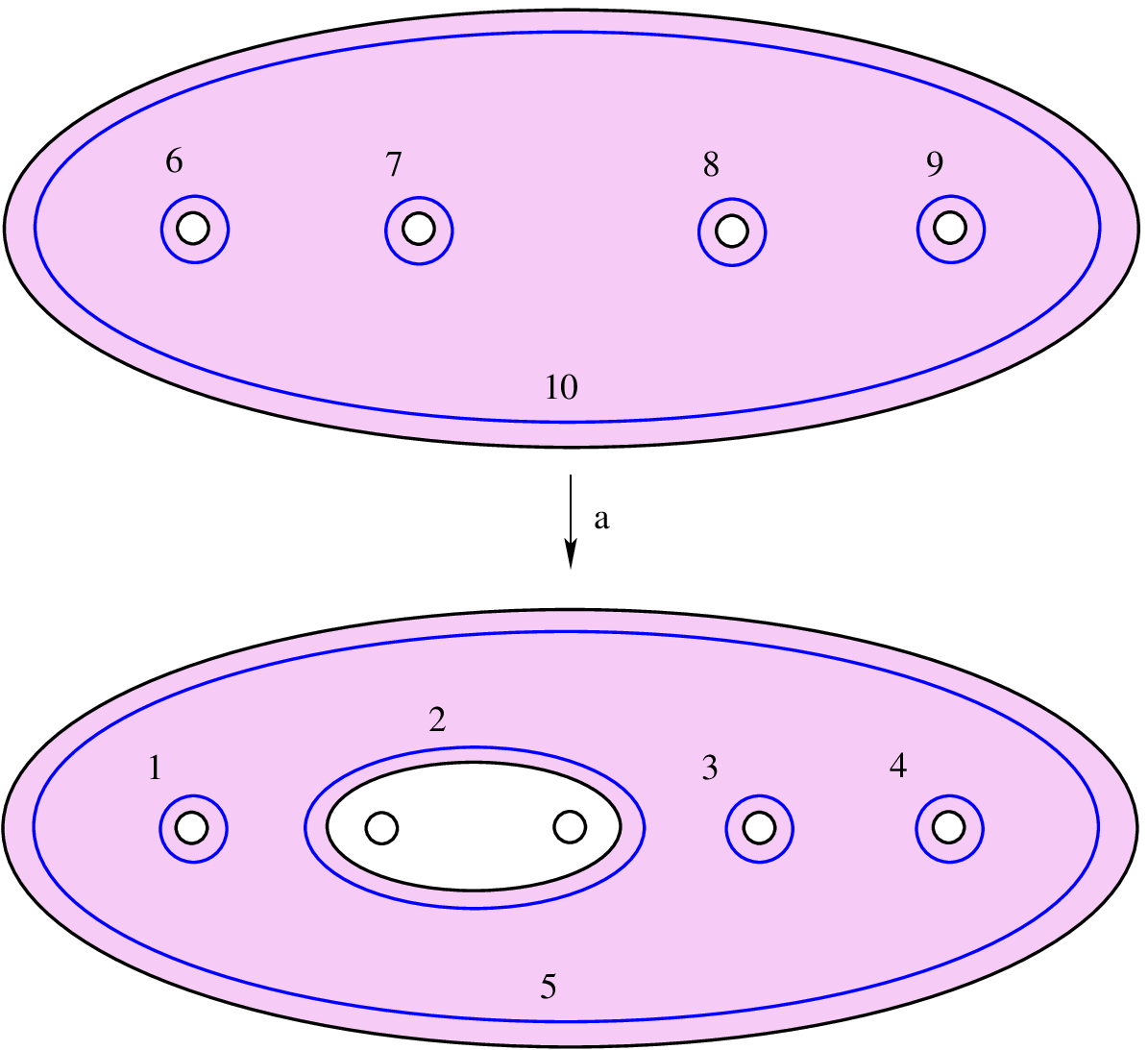}}}
 \relabel{1}{$\psi (\overline{\g}_1)$} \relabel{2}{$\psi (\overline{\a}_2)$} \relabel{3}{$\psi (\overline{\a}_3)$} \relabel{4}{$\psi (\overline{\a}_4)$} \relabel{5}{$\psi (\overline{\g}_4)$} \relabel{6}{$\overline{\g}_1$} \relabel{7}{$\overline{\a}_2$} \relabel{8}{$\overline{\a}_3$} \relabel{9}{$ \overline{\a}_4$} \relabel{10}{$\overline{\g}_4$} \relabel{a}{$\psi$}
\endrelabelbox
\caption{The embedding $\psi \colon D_4 \to D_5$.}\label{fig: embed} \end{figure}

More precisely, if $\psi \colon D_4 \to D_5$ denotes this embedding, then $\psi (\overline{\g}_1)=\g_1 $, $\psi (\overline{\a}_2)=\d_{2,3}$, $\psi (\overline{\a}_3)=\a_4$, $\psi (\overline{\a}_4)=\a_5$, and $\psi (\overline{\g}_4)=\g_5$. Therefore, the corresponding daisy substitution in $D_5$ replaces $$ D(\psi (\overline{\g}_1)) \circ D(\psi (\overline{\a}_2)) \circ  D(\psi (\overline{\a}_3)) \circ D(\psi (\overline{\a}_4))  \circ D^2(\psi (\overline{\g}_4))$$ by the product of four Dehn twists along curves in $D_5$ obtained by the embedding $\psi$.

\begin{Rem} \label{rem: conca}  {\em The concatenation of the two rational blowdowns $\rbd_1$ and  $\rbd_3$, is {\em not} a rational blowdown. To see this, we simply observe that  in the concatenation of the three lanterns (one for $\rbd_1$, and two for $\rbd_3$),  we use Dehn twists along the curves $\g_1, \g_1, \g_3, \g_3, \g_5, \g_5, \a_2, \a_3, \a_4, \a_5$ and  Dehn twists along the curves $\g_2, \g_3, \g_4$ emerge as a result of these substitutions. After the cancellations, we see that the product of Dehn twists along the curves $\g_1,  \g_1, \g_3, \g_5, \g_5, \a_2, \a_3, \a_4, \a_5$ is replaced by the product of Dehn twists along six curves in $D_5$. We claim that  this monodromy substitution {\em does not} correspond to a rational blowdown. Suppose, on the contrary, that it did correspond to a rational blowdown and consider
the contact lens space $(L(r,s),\xi_{can})$, where  $\frac{r}{s} = [2,4,4,2]$. Since  $\frac{r}{r-s}=[3,2,3,2,3]$, the monodromy of the canonical symplectic filling of $(L(r,s),\xi_{can})$ is a product of Dehn twists precisely along the curves $\g_1,  \g_1, \g_3, \g_5, \g_5, \a_2, \a_3, \a_4, \a_5$. As we are assuming that the given monodromy substitution does correspond to a rational blowdown, it follows that the
canonical symplectic filling of  $(L(r,s),\xi_{can})$ can be rationally blown down. However, this contradicts Proposition~\ref{prop: wahl}
as  $-2,-4,-4,-2$ cannot be obtained from $-4$ by iterations of type $(I)$ and $(II)$. Our claim follows.}
\end{Rem}

The minimal symplectic filling $(3,2,1,3,2)$ of height $3$, can be obtained from $(1,2,2,2,1)$ by a single rational blowdown $\rbd_4$ which is the composition of three lantern substitutions,
$$(1,2,2,2,1) \xrightarrow[\text{lantern}]{\text{$e_2$}}(1,\circled{3},\boxed{1},\circled{3},1) \xrightarrow[\text{lantern}]{\text{$e_{2,1}$}}  (\circled{2},\boxed{2},1,\circled{4},1) \xrightarrow[\text{lantern}]{\text{$e_{2,1,3}$}}  (\circled{3},2,1,\boxed{3},\circled{2}) $$  each of which is represented by the corresponding  blue edge  in  Figure~\ref{fig: graph8147}. The lantern corresponding to the blue edge $e_2$ requires the insertion of a cancelling pair of Dehn twists along $\g_2$, which is dictated by the fact that the triangulation $(1,3,1,3,1)$ that we skip  does not belong to $\mathcal{T}^{81,47} (P_6)$, exactly because its {\em second} entry (colored blue in Figure~\ref{fig: graph8147}) is one higher than the second entry of $(3,2,3,3,3)$. Then we apply a lantern substitution along the $4$-holed sphere bounded by the curves $\g_2, \a_3, \a_4, \g_4$.  As a result  of this first lantern substitution,  a  Dehn twist emerges  in the new factorization  along the convex curve $\d_{3,4}$.  After inserting a cancelling pair of Dehn twists along $\g_4$, which is dictated by the fact that the triangulation $(2,2,1,4,1)$ that we skip  does not belong to  $\mathcal{T}^{81,47} (P_6)$, exactly because its {\em fourth} entry (colored blue in Figure~\ref{fig: graph8147}) is one higher than the fourth entry of $(3,2,3,3,3)$. This allows one to apply a lantern substitution, corresponding to the blue edge $e_{2,1}$ along the $4$-holed sphere bounded by the curves $\g_1, \a_{2}, \d_{3,4}, \g_4$. In the  resulting monodromy factorization a Dehn twist along $\g_2$ emerges to cancel out the negative one we inserted in the previous step, but we still have a negative twist along $\g_4$. The final lantern substitution, corresponding to the blue edge $e_{2,1,3}$, is applied along the $4$-holed sphere bounded by the curves $\g_1, \d_{2,4}, \a_5, \g_5$. As a result  of this final lantern substitution,  a  Dehn twist emerges  in the new factorization  along  $\g_{4}$, which cancels out the negative one we inserted in the previous step so that we have an explicit  positive factorization for the minimal symplectic filling $(3,2,1,3,2)$.

In terms of the monodromy, by the concatenation of the $3$ lanterns, the product of Dehn twists along the curves $\g_1, \g_1, \g_4, \g_5, \a_2, \a_3, \a_4, \a_5$ is factorized into product of Dehn twists about $5$ other curves in $D_5$, which was explained in details in our paper \cite{bo}. Note that this is not a daisy substitution. Nevertheless,  we observe that the rational blowdown $\rbd_4$ is equivalent to replacing a neighborhood of the plumbing graph with weights  $-2,-5, -3$ with a rational homology ball, by Proposition~\ref{prop: lpg}, since $(3,2,1,3,2)$  is obtained from $(1,2,2,2,1)$ by flips along the contiguous sequence $d_2, d_1, d_3$ of distinguished diagonals in $(1,2,2,2,1)$. Nota that the rational blowdown $\rbd_4$ is not visible at all in the dual plumbing graph.

Finally,  we take a closer look at  $\rbd_5$ and $\rbd_6$. To realize  $\rbd_5$, corresponding to the edge $e_3$,  $$(1,2,2,2,1) \xrightarrow[\text{lantern}]{\text{$e_3$}}(1,2 ,\circled{3},\boxed{1},\circled{2}),$$ we apply the lantern substitution along the $4$-holed sphere bounded by the curves $\g_3, \a_4, \a_5, \g_5$, to get the monodromy factorization for the minimal symplectic filling $(1,2,3,1,2)$. Note that a Dehn twist along $\d_{4,5}$ appear in the new factorization. To  realize  $\rbd_6$, corresponding to the concatenation of the edge $e_{3,1}$, and $e_{3,1,2}$,  $$(1,2,3,1,2) \xrightarrow[\text{lantern}]{\text{$e_{3,1}$}}(\circled{2},\boxed{1} ,\circled{4},1,2)  \xrightarrow[\text{lantern}]{\text{$e_{3,1,2}$}}(\circled{3},1, \boxed{3},1, \circled{3}), $$ starting from the monodromy factorization for $(1,2,3,1,2)$, we insert a cancelling pair of Dehn twists along $\g_3$ and apply a lantern substitution along the $4$-holed sphere bounded by the curves $\g_1, \a_2, \a_3, \g_3$.  It follows that the  resulting
{\em achiral} monodromy factorization  corresponding to the triangulation $(2,1,4,1,2)$ is the same as described in the previous paragraph, and we proceed exactly in the same manner to finish the description of $\rbd_6$.

The concatenation of the two rational blowdowns $\rbd_5$ and  $\rbd_6$, is {\em not} a rational blowdown, which can be shown as in Remark~\ref{rem: conca}. In fact, we have the following result.

\begin{Prop}\label{prop: depth2a} The minimal symplectic filling $(3,1,3,1,3)$ of $(L(81, 47), \xi_{can})$ cannot be obtained from the canonical symplectic filling $(1,2,2,2,1)$, by a single symplectic rational blowdown. Consequently, the rational blowdown depth of the minimal symplectic filling $(3,1,3,1,3)$ is equal to $\dpt((3,1,3,1,3)) =2$.
\end{Prop}

\begin{proof} Note that any rational blowdown from the canonical symplectic filling $\textbf{u}_5=(1,2,2,2,1)$ of $(L(81, 47), \xi_{can})$  to the minimal symplectic filling  $(3,1,3,1,3)$ must be of height $3$, and hence requires a symplectic embedding of the linear plumbing graph with weights $-6,-2,-2$ or $-2,-5,-3$  into  $\textbf{u}_5$. To rule out any possible embedding of the linear plumbing graph with weights $-6,-2,-2$ into  $\textbf{u}_5$, we just show that the linear plumbing with weights $-2,-2$ cannot be embedded into $\textbf{u}_5$. On way to see the latter is as follows. We represent, as described in \cite[Section 2]{ggp}, the classes in $H_2 (\textbf{u}_5, \mathbb{Z})$ as nullhomologous linear combinations of the vanishing cycles of the planar Lefschetz fibration on $\textbf{u}_5$. Then we use the fact that any class whose square is $-2$ is a nullhomologous linear combination of only two vanishing cycles with $\pm 1$ coefficients to derive a contradiction.

Next, we discuss all possible symplectic embeddings of the linear plumbing graph with weights $-2,-5,-3$  into  $\textbf{u}_5$. Such an embedding is indeed possible because we just showed above that the minimal symplectic filling $(3,2,1,3,2)$ is obtained from $\textbf{u}_5$ by the (symplectic) rational blowdown $\rbd_4$, which is applied along  the linear plumbing graph with weights $-2,-5,-3$. We claim that there are only two possible symplectic embeddings of the linear plumbing graph with weights $-2,-5,-3$  into  $\textbf{u}_5$, and the intersection forms of the minimal symplectic fillings obtained by symplectic rational blowdowns along these embeddings are isomorphic. But since the intersection forms of the fillings $(3,2,1,3,2)$ and $(3,1,3,1,3)$, given respectively by 
$$
	\left[
		\begin{array}{rr}
			 -2	&	1	\\
			  1	&	 -5
\end{array}
		\right] \quad \mbox{and} \quad
\left[
		\begin{array}{rr}
			 -2	&	 1\phantom{1}	\\
			  1	&	 -41
\end{array}
		\right]
$$
are not isomorphic, we conclude that $(3,1,3,1,3)$ cannot be obtained by a symplectic rational blowdown from $\textbf{u}_5$ along any linear plumbing graph with weights $-2,-5,-3$. 

To finish the proof, we give an argument to prove our claim in the preceding paragraph as follows.  We notice  that the symplectic filling $\textbf{u}_5$ is diffeomorphic to the linear plumbing with weights $-2,-4,-3, -3, -2$  (since $\frac{81}{47}=[2,4,3,3,2]$)  and hence the lattice $H_2 (\textbf{u}_5, \mathbb{Z})$ is given by $$
	\left[
		\begin{array}{rrrrr}
			 -2	&	  1	& 0 &  0 &  0 \\
			  1	&	 -4	& 1 &  0 &  0 \\ 
            0	&	  1	& -3 &  1 &  0 \\
               0	&	  0	&  1 & -3 &  1 \\
               0	&	  0	&  0 &  1 & -2
\end{array}
		\right] $$ 
with respect to the natural basis represented by the symplectic spheres in the linear plumbing with weights $-2,-4,-3, -3, -2$, respectively.  Let $[S_1], \ldots, [S_5]$ denote these classes in $H_2 (\textbf{u}_5, \mathbb{Z})$.  By the adjunction equality we calculate that $c_1 ([S_1])= c_1 ([S_5])=0$, $c_1 ([S_2])= -2$,  $c_1 ([S_3])= c_1 ([S_4])=-1$, where $c_1([S_i])$ is defined to be $\langle c_1(\textbf{u}_5),[S_i]\rangle$.  Now using the lattice $H_2 (\textbf{u}_5, \mathbb{Z})$, and the restrictions induced by $c_1$, we conclude by a straightforward calculation, that the only embeddings of symplectic spheres of square $-2$, $-5$, $-3$ (as a linear chain in this order) are given by
$$ [S_1], [S_2]+[S_3], [S_4] \;\;\mbox{or} \;\; [S_1], [S_2]+[S_3], [S_4]+[S_5].$$  Finally, by considering the basis $[S_1], [S_2], [S_3],  [S_4]+[S_5],  -[S_5]$ of $H_2 (\textbf{u}_5, \mathbb{Z})$ , we see that these two embeddings are equivalent by a change of basis, and hence the intersection forms of the minimal symplectic fillings obtained by symplectic rational blowdowns along these embeddings are isomorphic.
\end{proof}

 \begin{figure}[ht]  \relabelbox \small {\epsfxsize=4in
\centerline{\epsfbox{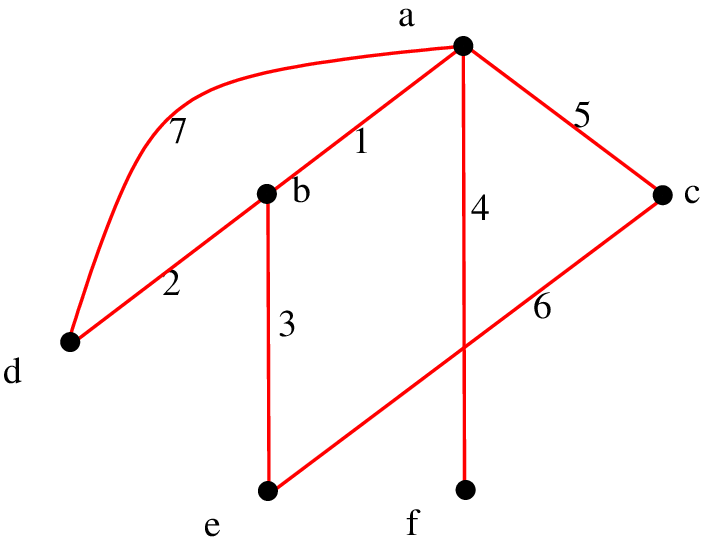}}}
 \relabel{a}{$(1,2,2,2,1)$}  \relabel{b}{$(2,1,3,2,1)$} \relabel{c}{$(1,2,3,1,2)$} \relabel{d}{$(3,1,2,3,1)$} \relabel{e}{$(3,1,3,1,3)$} \relabel{f}{$(3,2,1,3,2)$} \relabel{1}{$-4$} \relabel{2}{$-4$} \relabel{3}{$-5,-2$} \relabel{4}{$-2,-5,-3$} \relabel{5}{$-4$} \relabel{6}{$-2,-5$}  \relabel{7}{$-2,-5$}
\endrelabelbox
\caption{The rational blowdown graph $\mathcal{G}^{81, 47}_5$. }\label{fig: g8147} \end{figure}

 We  depicted another version of the graph $\mathcal{G}^{81, 47}_5$ in Figure~\ref{fig: g8147}, and decorated each edge with the weights of the linear plumbing graph of type $T_0$ that is used in the corresponding rational blowdown.

\subsection{Example C} \label{ex: C} Let $(p,q)=(37,10)$. Then $\frac{p}{p-q}=[2,2,3,2,4]$. The contact $3$-manifold $(L(37, 10), \xi_{can})$ has $4$ distinct minimal symplectic fillings, up to diffeomorphism, which are parametrized by the set $$\mathcal{Z}_{5}(\textstyle{\frac{37}{37-10}}) =  \{(1,2,2,2,1), (2,1,3,2,1), (1,2,3,1,2), (2,2,2,1,4)\}.$$ The graded, directed, rooted, connected graph $\mathcal{G}^{37, 10}_5$ can obtained from $\mathcal{G}_5$ as in {\bf Example B}. Note that the minimal symplectic filling $(2,1,3,2,1)$ can be obtained from $(1,2,2,2,1)$ by a rational blowdown ($\rbd_1$ of {\bf Example B}) along the edge $e_1$, while the minimal symplectic filling $(1,2,3,1,2)$ can be obtained from $(1,2,2,2,1)$ by a rational blowdown ($\rbd_5$ of {\bf Example B}) along the edge $e_3$.

The minimal symplectic filling $(2,2,2,1,4)$ can be obtained from $(1,2,3,1,2)$ by a rational blowdown obtained as a concatenation of the lantern substitutions corresponding to the edges $e_{3,2}$ and $e_{3,2,1}$  in  Figure~\ref{fig: hexagon1}.  Note that we need to insert a cancelling pair of Dehn twists along $\g_2$ for the lantern substitution corresponding to the edge $e_{3,2}$, but the negative Dehn twists cancels out once we apply the lantern substitution corresponding to the edge $e_{3,2,1}$.

 \begin{figure}[ht]  \relabelbox \small {\epsfxsize=3in
\centerline{\epsfbox{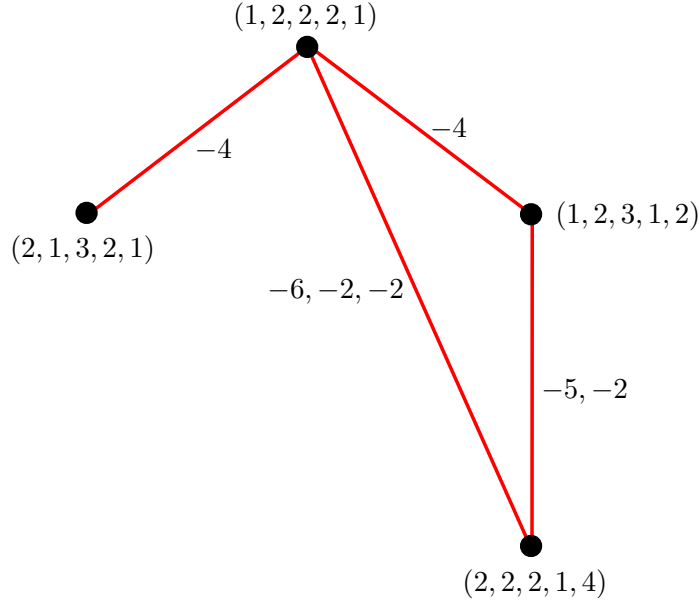}}}
 \relabel{a}{$(1,2,2,2,1)$}  \relabel{b}{$(2,1,3,2,1)$} \relabel{c}{$(1,2,3,1,2)$} \relabel{d}{$(2,2,2,1,4)$} \relabel{1}{$-4$} \relabel{6}{$-5,-2$} \relabel{5}{$-4$} \relabel{4}{$-6,-2,-2$}
\endrelabelbox
\caption{The rational blowdown graph $\mathcal{G}^{37, 10}_5$.}\label{fig: g3710} \end{figure}

Moreover, the minimal symplectic filling $(2,2,2,1,4)$ can be obtained from the minimal resolution $(1,2,2,2,1)$ by a {\em single} rational blowdown which is the concatenation of three lantern substitutions,
$$(1,2,2,2,1) \xrightarrow[\text{lantern}]{\text{$e_3$}}(1,2,\circled{3},\boxed{1},\circled{2}) \xrightarrow[\text{lantern}]{\text{$e_{3,2}$}}  (1, \circled{3},\boxed{2},1,\circled{3}) \xrightarrow[\text{lantern}]{\text{$e_{3,2,1}$}}  (\circled{2},\boxed{2},2,1,\circled{4}). $$  We observe that  in the concatenation of $3$ lanterns we use Dehn twists along the curves $\g_1, \g_2, \g_3, \g_5, \g_5, \g_5, \a_2, \a_3, \a_4, \a_5$ and  Dehn twists along the curves $\g_2, \g_3, \g_4$ emerge as a result of these substitutions. After the cancelations, we only use Dehn twists along the curves $\g_1, \g_5, \g_5, \g_5, \a_2, \a_3, \a_4, \a_5$. This means that, there is a relation in the mapping class group of the $6$-holed sphere, where the product $$D(\g_1) \circ D(\a_2) \circ D(\a_3) \circ D(\a_4) \circ D(\a_5) \circ D^3(\g_5)$$ is isotopic to the product of Dehn twists about $5$ curves, which is precisely the {\em daisy relation} for the $6$-holed sphere. Since the triangulation $(2,2,2,1,4)$ is obtained from $(1,2,2,2,1)$ by flips along the contiguous sequence $d_3, d_2, d_1$ of distinguished diagonals in $(1,2,2,2,1)$, the linear plumbing graph for this rational blowdown  has weights  $-6,-2,-2$, by Proposition~\ref{prop: lpg}.

\subsection{Example D} \label{ex: D} Let $(p,q)=(45,26)$. Then $\frac{p}{p-q}=[3,2,3,2,3]$. The contact $3$-manifold $(L(45, 26), \xi_{can})$ has $4$ distinct minimal symplectic fillings, up to diffeomorphism, which are parametrized by the set $$\mathcal{Z}_{5}(\textstyle{\frac{37}{37-10}}) =  \{(1,2,2,2,1), (2,1,3,2,1), (1,2,3,1,2), (3,1,3,1,3)\}.$$  The graph $\mathcal{G}^{45, 26}_5$ is a subgraph of $\mathcal{G}^{81, 47}_5$ which we depicted in  Figure~\ref{fig: g8147},  and the discussion in {\bf Example B}, applies verbatim here.

\begin{Rem} {\em The interested reader may compare {\bf Examples B, C, D} above, with the examples in \cite[Section 4.1]{cps}, where they use sequences of rational blowdowns and {\em symplectic antiflips} to obtain minimal symplectic fillings from the minimal resolution.}
\end{Rem}

\subsection{Example E} \label{ex: E} Let $(p,q)=(140,41)$. Then $\frac{p}{p-q}=[2,2,4,2,4,2,2]$. The contact $3$-manifold $(L(140, 41), \xi_{can})$ has $8$ distinct minimal symplectic fillings, up to diffeomorphism, which are parametrized by the set $\mathcal{Z}_{7}(\textstyle{\frac{140}{140-41}}) =$ $$ \{(1,2,2,2,2,2,1), (2,1,3,2,2,2,1), (1,2,3,1,3,1,2), (1,2,2,2,3,1,2) $$  $$(2,1,4,1,3,2,1), (2,1,3,2,3,1,2), (1,2,3,1,4,1,2), (2,1,4,1,4,1,2) \}. $$ The graph $\mathcal{G}^{140, 41}_7$ is depicted in  Figure~\ref{fig: g14041}.

 \begin{figure}[ht]  \relabelbox \small {\epsfxsize=5.5in
\centerline{\epsfbox{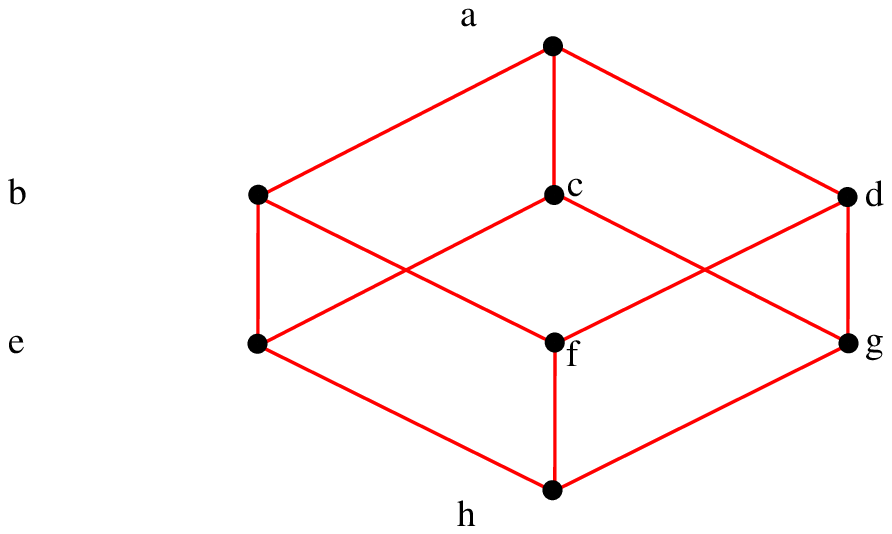}}}
 \relabel{a}{$(1,2,2,2,2,2,1)$}  \relabel{b}{$(2,1,3,2,2,2,1)$} \relabel{c}{$(1,2,3,1,3,1,2)$} \relabel{d}{$(1,2,2,2,3,1,2) $}  \relabel{e}{$(2,1,4,1,3,2,1)$}  \relabel{f}{$(2,1,3,2,3,1,2)$} \relabel{g}{$(1,2,3,1,4,1,2)$} \relabel{h}{$(2,1,4,1,4,1,2) $}
\endrelabelbox
\caption{The rational blowdown graph $\mathcal{G}^{140, 41}_7$.}\label{fig: g14041} \end{figure}

We claim that there are no symplectic rational blowdowns between these $8$ symplectic fillings other than the ones represented  by the edges in Figure~\ref{fig: g14041}. To prove our claim, we first show that there is no height $3$ rational blowdown from the canonical symplectic filling $\textbf{u}_7 =  (1,2,2,2,2,2,1)$. To see this, we first observe that the symplectic filling $\textbf{u}_7$ is diffeomorphic to the linear plumbing with weights $-4,-2,-4, -2, -4$  (since $\frac{p}{q}=[4,2,4,2,4]$)  and hence the lattice $H_2 (\textbf{u}_7, \mathbb{Z})$ is even. Thus, the linear plumbing with weights $-2,-5,-3$ cannot be embedded into $\textbf{u}_7$. Next we rule out any possible embedding of the linear plumbing with weights $-6,-2,-2$ into $\textbf{u}_7$ by showing that the linear plumbing with weights $-2,-2$ cannot be embedded into $\textbf{u}_7$. On way to see the latter is as follows. We represent, as described in \cite[Section 2]{ggp}, the classes in $H_2 (\textbf{u}_7, \mathbb{Z})$ as nullhomologous linear combinations of the vanishing cycles of the planar Lefschetz fibration on canonical symplectic filling $\textbf{u}_7$. Then we use the fact that any class whose square is $-2$ is nullhomologous linear combination of only two vanishing cycles with $\pm 1$ coefficients to derive a contradiction.

It is easy to see that there are no height $2$ rational blowdowns from the canonical symplectic filling $\textbf{u}_7 $ since $H_2 (\textbf{u}_7, \mathbb{Z})$ is even and hence the linear plumbing with weights $-2,-5$ cannot be embedded into $\textbf{u}_7$.

According to the classification of minimal symplectic fillings, there are only $3$ distinct fillings, namely $(2,1,3,2,2,2,1), (1,2,3,1,3,1,2), (1,2,2,2,3,1,2)$, which can be obtained from $\textbf{u}_7$ by a single symplectic rational blowdown. In fact, these correspond to the $3$ distinct $-4$ curves in $\textbf{u}_7$. We now show that none of these $3$ fillings contains an embedded linear plumbing with weights $-2,-5$.
The lattices for the symplectic fillings, $(2,1,3,2,2,2,1), (1,2,3,1,3,1,2), (1,2,2,2,3,1,2)$ are given, respectively,  by
$$\left[
\begin{array}{rrrr}
-7 & 2 & 0 & 0 \\
2 & -4 & 1 & 0 \\
0 & 1 & -2 & 1 \\
0 & 0 & 1 & -4
\end{array}
		\right], \quad
		 \left[
\begin{array}{rrrr}
-4 & 1 & 0 & 0 \\
 1 & -4 & -4 & -1 \\
0 & -4 & -7 & -2 \\
0 & -1 & -2 & -4
\end{array}
		\right], \quad
		 \left[
\begin{array}{rrrr}
-4 & 1& 0 & 0 \\
1 & -2 & 1 & 0 \\
0 & 1 & -4 & 2 \\
0 & 0 & 2 & -7
\end{array}
		\right]. $$ 
One can check that there is no class of square $-5$ in the first and the third lattices, whereas there is no class of square $-2$ in the second lattice.

In particular, we obtain the following result. 

\begin{Prop}\label{prop: depth3} The rational blowdown depth of the minimal symplectic filling $(2,1,4,$ $1,4,1,2)$  of $(L(140, 41), \xi_{can})$ is equal to $\dpt((2,1,4,1,4,1,2)) =3$.
\end{Prop}


\begin{thebibliography}{99999}


\bibitem{bono}  M. Bhupal and K. Ono, {\em Symplectic fillings of links of quotient surface singularities.} Nagoya Math. J. 207 (2012), 1-45.

\bibitem{bo}  M. Bhupal and B. Ozbagci, {\em  Symplectic fillings of lens spaces as Lefschetz fibrations.} J. Eur. Math. Soc. (JEMS) 18 (2016), no. 7, 1515-1535.

\bibitem{bd} F. A. Bogomolov and B.  de Oliveira, \emph{minimal symplectic small
    deformations of strictly pseudoconvex
    surfaces.} Birational algebraic geometry (Baltimore, MD, 1996),  25--41, Contemp. Math., 207, Amer. Math.
    Soc., Providence, RI, 1997.

 \bibitem{c} J. A. Christophersen, {\em On the components and discriminant of the versal base space of cyclic quotient
singularities.} Singularity theory and its applications, Part I, (eds D. Mond, J. Montaldi), Lecture Notes
in Mathematics 1462 (Springer, Berlin, 1991) 81-92.

\bibitem{cps} H. Choi, H. Park, and D. Shin,  {\em Symplectic fillings of quotient surface singularities and minimal model program.}
J. Korean Math. Soc. 58 (2021), no. 2, 419-437.

\bibitem{emvhm} H. Endo, T. E. Mark,  and J. Van Horn-Morris, {\em Monodromy substitutions and rational blowdowns.} J. Topol.  4  (2011),  no. 1, 227-253.

\bibitem{en} H. Endo and S. Nagami, {\em Signature of relations in mapping class groups and non-holomorphic Lefschetz fibrations.} Trans. Am. Math. Soc. 357 (8) (2005) 3179-3199.

\bibitem{fs} R. Fintushel and R. J. Stern, {\em Rational blowdowns of 4-manifolds}, J. Diff. Geom. 46 (1997), 181--235.

\bibitem{gm} D. Gay and T. E. Mark, {\em Convex plumbings and Lefschetz fibrations.} J. Symplectic Geom. 11 (2013), no. 3, 363-375.

\bibitem{ggp} P. Ghiggini, M. Golla, and O. Plamenevskaya, {\em Surface singularities and planar contact structures.} Ann. Inst. Fourier (Grenoble) 70 (2020), no. 4, 1791-1823.

\bibitem{ksb} J. Koll\'{a}r, and N. I. Shepherd-Barron, {\em Threefolds and deformations of surface singularities.}
Invent. Math. 91 (1988), no. 2, 299-338.


\bibitem{l} P. Lisca, {\em On symplectic fillings of lens spaces.} Trans. Amer. Math. Soc. 360 (2008), no. 2, 765-799.

\bibitem{lw} E. Looijenga and J. Wahl, {\em Quadratic functions and smoothing surface singularities.}  Topology 25 (1986), no. 3, 261-291.

\bibitem{n} A.  N\'{e}methi, {\em Some meeting points of singularity theory and low dimensional topology.} Deformations of surface singularities, 109-162, Bolyai Soc. Math. Stud., 23, J\'{a}nos Bolyai Math. Soc., Budapest, 2013.

\bibitem{npp} A. N\'{e}methi and P. Popescu-Pampu, {\em On the Milnor fibres of cyclic quotient singularities.} Proc. Lond. Math. Soc. (3) 101 (2010), no. 2, 554-588.

\bibitem{p} J. Park, {\em Seiberg-Witten invariants of generalised rational blow-downs}, Bull. Austral. Math. Soc.  56  (1997),  no. 3, 363--384.

\bibitem{r} O. Riemenschneider, {\em Deformationen von Quotientensingularit\"{a}ten (nach zyklischen Gruppen).} Math.
Ann. 209 (1974) 211-248.

\bibitem{s} J. Stevens, {\em On the versal deformation of cyclic quotient singularities.} Singularity theory and its
applications, Part I (eds D. Mond and J. Montaldi), Lecture Notes in Mathematics 1462 (Springer, Berlin,
1991) 312-319.

\bibitem{sy1} M. Symington,  {\em  Symplectic rational blowdowns.}  J. Differential Geom. 50 (1998), no. 3, 505-518.

\bibitem{sy2} M. Symington,  {\em  Generalized symplectic rational blowdowns.}  Algebr. Geom. Topol. 1 (2001), 503-518.

\bibitem{wa} J. M. Wahl, {\em Smoothings of normal surface singularities.} Topology 20 (1981), no. 3, 219-246.

\bibitem{w} C. Wendl, {\em Strongly fillable contact manifolds and J -holomorphic foliations}, Duke Math. J.  151  (2010),  no. 3, 337-384.



\end{thebibliography}
\end{document}